\documentclass[12pt,reqno]{amsart}
\usepackage{amssymb}
\usepackage{mathrsfs}
\usepackage{amsmath}
\usepackage{comment}
\usepackage{enumerate}
\usepackage[latin1]{inputenc}
\usepackage[colorlinks=true, linkcolor=blue, citecolor=blue]{hyperref}

{\catcode`p =12 \catcode`t =12 \gdef\eeaa#1pt{#1}}      %
\def\accentadjtext#1{\setbox0\hbox{$#1$}\kern   %
                \expandafter\eeaa\the\fontdimen1\textfont1 \ht0 }
\def\accentadjscript#1{\setbox0\hbox{$#1$}\kern %
                \expandafter\eeaa\the\fontdimen1\scriptfont1 \ht0 }
\def\accentadjscriptscript#1{\setbox0\hbox{$#1$}\kern   %
                \expandafter\eeaa\the\fontdimen1\scriptscriptfont1 \ht0 }
\def\accentadjtextback#1{\setbox0\hbox{$#1$}\kern       %
                -\expandafter\eeaa\the\fontdimen1\textfont1 \ht0 }
\def\accentadjscriptback#1{\setbox0\hbox{$#1$}\kern     %
                -\expandafter\eeaa\the\fontdimen1\scriptfont1 \ht0 }
\def\accentadjscriptscriptback#1{\setbox0\hbox{$#1$}\kern %
                -\expandafter\eeaa\the\fontdimen1\scriptscriptfont1 \ht0 }
\def\itoverline#1{{\mathsurround0pt\mathchoice
        {\rlap{$\accentadjtext{\displaystyle #1}
                \accentadjtext{\vrule height1.593pt}
                \overline{\phantom{\displaystyle #1}
                \accentadjtextback{\displaystyle #1}}$}{#1}}
        {\rlap{$\accentadjtext{\textstyle #1}
                \accentadjtext{\vrule height1.593pt}
                \overline{\phantom{\textstyle #1}
                \accentadjtextback{\textstyle #1}}$}{#1}}
        {\rlap{$\accentadjscript{\scriptstyle #1}
                \accentadjscript{\vrule height1.593pt}
                \overline{\phantom{\scriptstyle #1}
                \accentadjscriptback{\scriptstyle #1}}$}{#1}}
        {\rlap{$\accentadjscriptscript{\scriptscriptstyle #1}
                \accentadjscriptscript{\vrule height1.593pt}
                \overline{\phantom{\scriptscriptstyle #1}
                \accentadjscriptscriptback{\scriptscriptstyle #1}}$}{#1}}}}

\newcommand{\iol}{\itoverline}

\newcommand{\ch}[1]{{\mbox{\raise 1pt\hbox{\large$\chi$}}}_{\lower 1pt\hbox{$\scriptstyle #1$}}}

\usepackage[left=2.1cm,top=2.3cm,right=2.1cm]{geometry}

\def\1{\raisebox{2pt}{\rm{$\chi$}}}

\newtheorem{theorem}{Theorem}[section]
\newtheorem{corollary}[theorem]{Corollary}
\newtheorem{lemma}[theorem]{Lemma}

\theoremstyle{definition}
\newtheorem{definition}[theorem]{Definition}

\newtheorem{example}[theorem]{Example}

\DeclareFontFamily{U}{mathx}{}
\DeclareFontShape{U}{mathx}{m}{n}{<-> mathx10}{}
\DeclareSymbolFont{mathx}{U}{mathx}{m}{n}
\DeclareMathAccent{\widehat}{0}{mathx}{"70}
\DeclareMathAccent{\widecheck}{0}{mathx}{"71}

\newcommand{\R}{{\mathbb R}}

\newcommand{\N}{{\mathbb N}}
\newcommand{\Z}{{\mathbb Z}}

\newcommand{\Ha}{{\mathcal H}}

\newcommand\cp{\operatorname{cap}}
\newcommand\Cp{\operatorname{Cap}}
\newcommand{\capm}[3]{\cp_{\dot{M}^{#1}_{#2,#3}}}

\newcommand\diam{\operatorname{diam}}

\DeclareMathOperator{\Hop}{\mathcal{H}}

\DeclareMathOperator{\Lop}{{\mathcal{Q}}}

\def\1{\raisebox{2pt}{\rm{$\chi$}}}

\def\vint_#1{\mathchoice%
        {\mathop{\kern 0.2em\vrule width 0.6em height 0.69678ex depth -0.58065ex
                \kern -0.8em \intop}\nolimits_{\kern -0.4em#1}}%
        {\mathop{\kern 0.1em\vrule width 0.5em height 0.69678ex depth -0.60387ex
                \kern -0.6em \intop}\nolimits_{#1}}%
        {\mathop{\kern 0.1em\vrule width 0.5em height 0.69678ex
            depth -0.60387ex
                \kern -0.6em \intop}\nolimits_{#1}}%
        {\mathop{\kern 0.1em\vrule width 0.5em height 0.69678ex depth -0.60387ex
                \kern -0.6em \intop}\nolimits_{#1}}}
\def\vintslides_#1{\mathchoice%
        {\mathop{\kern 0.1em\vrule width 0.5em height 0.697ex depth -0.581ex
                \kern -0.6em \intop}\nolimits_{\kern -0.4em#1}}%
        {\mathop{\kern 0.1em\vrule width 0.3em height 0.697ex depth -0.604ex
                \kern -0.4em \intop}\nolimits_{#1}}%
        {\mathop{\kern 0.1em\vrule width 0.3em height 0.697ex depth -0.604ex
                \kern -0.4em \intop}\nolimits_{#1}}%
        {\mathop{\kern 0.1em\vrule width 0.3em height 0.697ex depth -0.604ex
                \kern -0.4em \intop}\nolimits_{#1}}}

\newcommand{\intav}{\vint}

\newcommand{\dist}{\operatorname{dist}}

\title[Triebel--Lizorkin capacities in metric spaces]{Classifying  Triebel--Lizorkin capacities in metric spaces}

\author[J. Lehrb\"ack]{Juha Lehrb\"ack}   %
\address[J.L.]{Department of Mathematics and Statistics, P.O. Box 35, FI-40014 University of Jyvaskyla, Finland}
\email{juha.lehrback@jyu.fi}

\author[K. Mohanta]{Kaushik Mohanta}   %
\address[K.M.]{Department of Mathematics and Statistics, P.O. Box 35, FI-40014 University of Jyvaskyla, Finland}
\email{kaushik.k.mohanta@jyu.fi}

\author[A. V. V\"ah\"akangas]{Antti V. V\"ah\"akangas}
\address[A.V.V.]{Department of Mathematics and Statistics, P.O. Box 35, FI-40014 University of Jyvaskyla, Finland}
 \email{antti.vahakangas@iki.fi}

\makeatletter
\@namedef{subjclassname@2020}{{\mdseries 2020} Mathematics Subject Classification}
\makeatother



\begin{document}

\begin{abstract} 
We study non-local  or fractional capacities in metric measure spaces.
Our main goal is to clarify the relations between 
relative Haj{\l}asz--Triebel--Lizorkin capacities,
potentional Triebel--Lizorkin capacities, and 
metric space variants of Riesz capacities.
As an application of our results, we obtain a characterization of a 
Haj{\l}asz--Triebel--Lizorkin capacity density condition,
which is based on an earlier characterization of a 
Riesz capacity density condition in terms of
Hausdorff contents.
\end{abstract}

\maketitle

\section{Introduction}

In mathematical analysis, there are several possible approaches to the concept of 
\emph{capacity of a set}, depending, for instance, on the
general setting and the intended applications. Nevertheless, it often turns out that 
seemingly different definitions still lead to comparable or
otherwise closely related concepts.

In this work, we examine and compare four different
approaches to nonlocal or fractional capacities in a
metric space $X$ equipped with a doubling measure $\mu$.
More precisely, when
$B\subset X$ is an open ball and $E\subset\iol B$ is a closed set, 
we consider the following, 
for parameters $1<p<\infty$, $1\le q \le\infty$, $0<\beta< 1$, and $\Lambda\ge 2$:
\begin{itemize}
\item the relative Haj{\l}asz--Triebel--Lizorkin $(\beta,p,q)$-capacity $\capm{\beta}{p}{q}(E,2B,\Lambda B)$,
\item the potentional Triebel--Lizorkin  $(\beta,p,q)$-capacity $\Cp^{\beta}_{p,q}(E)$, 
\item the dual characterization of $\Cp^{\beta}_{p,q}(E)$ in terms of measures supported on $E$, 
\item the Riesz $(\beta,p)$-capacity $R_{\beta,p}(E)$.
\end{itemize}
The main purpose of this paper is to establish and clarify the relations between these
concepts. For instance, we examine when these capacities are comparable  and 
show that often the parameter $q$ plays no essential role.
 The last part of the paper, where
we characterize a capacity density condition associated to Haj{\l}asz--Triebel--Lizorkin capacities,
can be viewed as a continuation of \cite{CILV,MR4478471}. 

The Haj{\l}asz--Triebel--Lizorkin function spaces  $\dot{M}^\beta_{p,q}$  and the corresponding relative capacities $\capm{\beta}{p}{q}$ 
are defined using fractional Haj{\l}asz gradients. 
In Section~\ref{s.frac_haj} we introduce these gradients, following~\cite{MR2764899},
and recall in Lemma~\ref{lemma: Poincare} a Poincar\'e type inequality for them.
Haj{\l}asz--Triebel--Lizorkin spaces and relative capacities $\capm{\beta}{p}{q}$ are
defined in Section~\ref{s.tl_space}, where we also
prove a lower bound estimate for these capacities in terms of  suitable Hausdorff contents; 
see Theorem~\ref{t.HC_bp}.

On the other hand, 
following~\cite{MR1097178,MR1411441},
the Triebel--Lizorkin potentional capacity $\Cp^{\beta}_{p,q}(E)$ is 
defined in Section~\ref{s.pot_cap} using the potential operator $\Hop_\beta$,
\[
\Hop_\beta f(x)= \sum_{n=n_0}^\infty  2^{-\beta n}\int_{X} \frac{\mathbf{1}_{B(y,2^{-n})}(x)}{\mu(B(y,2^{-n}))} f_n(y)\,d\mu(y)\,,
\] 
for a sequence $f=(f_n)_{n\ge n_0}$ of nonnegative Borel functions in $X$.
Here $n_0$ is the smallest integer
such that $2^{-n_0}\le 2\diam(X)$, with the interpretation that $n_0=-\infty$ if $X$ is unbounded.
 
We show that, under suitable assumptions on the space $X$ 
and the parameters
$p$, $q$, $\beta$ and $\Lambda$, 
the capacities 
$\capm{\beta}{p}{q}(E,2B,\Lambda B)$ and $\Cp^{\beta}_{p,q}(E)$ are comparable
for all compact sets $E\subset \iol{B}$.
In particular, besides the doubling condition (see~\eqref{e.doubling}),
we need to assume that the measure $\mu$ is also reverse 
doubling (see~\eqref{e.rev_dbl_decay} and \eqref{e.reverse_doubling}), 
as follows.

\begin{theorem}\label{t.equiv_intro}
Let $1<p<\infty$, $1<q\le\infty$, $0<\beta <1$ and $\Lambda\ge  41$.
Assume that the measure $\mu$ is doubling and reverse doubling,  
and satisfies the quantitative reverse doubling condition~\eqref{e.reverse_doubling}  
for some exponent $\sigma>\beta p$.
Let $B=B(x,r)\subset X$ be a ball 
with $x\in X$ and $0<r<\frac{1}{80}\diam(X)$ 
and let $E\subset \iol{B}$ be a compact set.
Then 
\begin{equation}\label{e.equiv_intro}
C^{-1} \Cp^\beta_{p,q}(E) \le \capm{\beta}{p}{q}(E, 2B,\Lambda B)\le  C \Cp^\beta_{p,q}(E)\,,
\end{equation}
where the constant $C\ge 1$ is independent of $B$ and $E$.
\end{theorem}

The two inequalities in~\eqref{e.equiv_intro}
hold separately under slightly different assumptions, 
see Theorems~\ref{t.equiv_bp} and~\ref{t.comparison_dc} for 
precise formulations.
Also recall, as a basic example, that in $\R^n$  (equipped with the 
Euclidean distance) the Lebesgue measure is both doubling and reverse doubling and 
satisfies  condition~\eqref{e.reverse_doubling} with $\sigma=n$.

As a technical tool for the proof of the second inequality in~\eqref{e.equiv_intro},
we introduce a new discrete potential operator $\Lop_\beta$, 
which majorizes the operator $\Hop_\beta$; see Lemma~\ref{l.disc_est}. 
The definition of $\Lop_\beta$ (Definition~\ref{d.dpotential})
involves a partition of unity and
is inspired by the so-called discrete convolutions, see 
e.g.~\cite{MR2747071,MR1954868}. 
Since the functions in the partition of unity are
Lipschitz-continuous, the operator $\Lop_\beta$ can be seen as a
smooth version of $\Hop_\beta$, whose definition 
uses noncontinuous characteristic functions of balls.
Boundedness properties of $\Lop_\beta$ 
are studied in Section~\ref{s.discr_pot},
and in particular we obtain  for $\Lop_\beta$ a mixed norm estimate in 
Lemma~\ref{l.pot_norm_est} and a 
Sobolev type inequality in Lemma~\ref{l.pot_poincare}.

In Section~\ref{s.Riesz_capacity}, we prove a dual characterization of the
Triebel--Lizorkin potentional capacity $\Cp^{\beta}_{p,q}(E)$
for compact sets $E\subset X$.
As in the Euclidean case in~\cite{MR1411441},
this is based on the use of
(nonnegative) Radon measures $\nu$ supported on $E$.
In our setting, the sequence
\begin{equation*}%
\widecheck{\mathcal{H}}_\beta\nu(x)=\left(  2^{-\beta n} \frac{\nu(B(x,2^{-n}))}{\mu(B(x,2^{-n}))} \right)_{n=n_0}^\infty\,,\qquad x\in X\,,
\end{equation*}
plays an important role.
The main tool in the proof of the dual characterization (Theorem~\ref{t.n_cap_dual}) is a Minimax theorem, see
Lemma~\ref{l.m_assumptions}. 

In Section~\ref{s.MW} we prove a version of the
Muckenhoupt--Wheeden theorem  (Theorem~\ref{t.M-W}), which gives an estimate for the 
norm of $\widecheck{\mathcal{H}}_\beta\nu$
in terms of the fractional maximal function
$$
M_\beta\nu(x)=\sup_{0<r\le 2^{-n_0}} r^\beta \frac{\nu(B(x,r))}{\mu(B(x,r))}\,,\qquad x\in X\,.
$$
Such results originate in \cite{MR340523}, and a very general metric space variant 
of the Muckenhoupt--Wheeden theorem can be found in~\cite{MR1962949}.
Nevertheless, instead of relying on the more abstract machinery from~\cite{MR1962949}, we present an alternative
self-contained proof that is suited for our purposes.

The dual characterization of $\Cp^{\beta}_{p,q}$
and the Muckenhoupt--Wheeden theorem are crucial ingredients in the proof of Theorem~\ref{t.cnr_ML}, 
which shows that for compact sets  the Triebel--Lizorkin capacities $\Cp^{\beta}_{p,q}$ are in fact 
independent of the $q$-parameter, again under a suitable reverse doubling assumption.
Consequently, by Theorem~\ref{t.equiv_intro} the same independence holds for the relative capacity
$\capm{\beta}{p}{q}$, as well.

Finally, in Sections~\ref{s.comparison} and~\ref{s.density}
we compare the Triebel--Lizorkin capacities and the Riesz capacities $R_{\beta,p}(E)$,
which are defined in terms of a metric space variant of the classical Riesz potential, 
\begin{equation}\label{e.riesz_intro}
I_\beta f(x) %
=\int_{X\setminus \{x\}}\frac{d(x,y)^\beta}{\mu(B(x,d(x,y)))}f(y)\,d\mu(y)\,,
\end{equation}
for Borel functions $f\ge 0$ in $X$.
See Definitions~\ref{d.RieszPot} and~\ref{d.RieszCap} for more details.

In Theorem~\ref{t.infinity} we show that
$\Cp^\beta_{p,\infty}$ and $R_{\beta,p}$ are comparable. 
The proof of this theorem follows by rather straightforward 
applications of the doubling and
quantitative reverse doubling conditions, see Lemma~\ref{l.riesz_equiv}.
The $q$-independence of the
potentional Triebel--Lizorkin capacities yields  
for compact sets the same comparability 
also for $\Cp^\beta_{p,q}$,
for every $1<q\le \infty$,
and thus for the relative capacity $\capm{\beta}{p}{q}$ as well, by 
Theorem~\ref{t.equiv_intro}.
In Theorem~\ref{t.cnr} we present an alternative direct proof,
again based on the dual characterization, for the fact that
for compact sets
$\Cp^\beta_{p,q}$ is bounded from above by $R_{\beta,p}$, for all  
$1<q\le \infty$.

The main application of the above results is given in Section~\ref{s.density}, where we
characterize, for $0<\beta<1$, a Haj{\l}asz--Triebel--Lizorkin capacity density condition
in complete geodesic metric spaces.
The proof of the characterization is based on the
following facts:
\begin{enumerate}
\item The Riesz capacity $R_{\beta,p}$ dominates the corresponding  
Haj{\l}asz--Triebel--Lizorkin capacity $\capm{\beta}{p}{q}$, for every $1<q\le\infty$;
see Corollary~\ref{c.equiv_bp_SEC}.
\item The Riesz capacity density condition is equivalent to 
a Hausdorff content density condition. This is a deep and highly nontrivial result 
from~\cite[Theorem~6.4]{CILV} and~\cite[Theorem~9.5]{MR4478471},
and in particular implies a self-improvement result for the Riesz capacity density condition.
Here the additional assumptions that the space $X$ is complete and geodesic are needed.
\item The Haj{\l}asz--Triebel--Lizorkin capacity $\capm{\beta}{p}{q}$ can be estimated from below by suitable
Hausdorff contents, by Theorem~\ref{t.HC_bp}.
\end{enumerate}
In addition to these main ingredients, we also need estimates for the 
capacities and contents of balls. For Haj{\l}asz--Triebel--Lizorkin capacities
such estimates are obtained in Theorem~\ref{t.cap_balls}. 

The above characterization leads to a self-improvement result (Corollary \ref{c.riesz-improvement}) for the 
Haj{\l}asz--Triebel--Lizorkin capacity density condition, which is,
to the best of our knowledge, new even in the context of
Euclidean spaces.
Self-improvement of the capacity density condition for the Riesz capacity in $\R^n$
was originally proven by Lewis~\cite{MR946438}. See also \cite{CILV,MR4478471}
for corresponding results for other capacities in metric spaces.

Before going to the main parts of the paper, we begin in Section~\ref{s.prelim}
by introducing notation and recalling some useful concepts that are
needed later on.

\subsection*{Acknowledgements} The authors would like to thank Lizaveta Ihnatsyeva 
for interesting discussions on the topics of the paper.
The second author was supported by the Academy of Finland (project \#323960) and by the Academy of Finland via Centre of Excellence in Analysis and Dynamics Research (project \#346310).

\section{Preliminaries}\label{s.prelim}

We write $\N=\{1,2,3,\ldots\}$
and $\N_0=\{0,1,2,\ldots\}$.
We also set $p'=\frac{p}{p-1}$ if $1<p<\infty$, $p'=1$ if $p=\infty$, 
and $p'=\infty$ if $p=1$.

Throughout the rest of the paper we assume that $X=(X,d,\mu)$ is a metric 
measure space equipped with a metric $d$ and a 
nonnegative complete Borel measure $\mu$ such that $0<\mu(B)<\infty$
for all open balls $B\subset X$, 
each of which is an open set of the form 
\[
B=B(x,r)=\{y\in X\,:\, d(y,x)<r\}
\] 
with $x\in X$ and $r>0$.
The corresponding closed ball is
\[
\iol{B}(x,r)=\{y\in X\,:\, d(y,x)\le r\}.
\] 
Under these assumptions,
the space $X$ is separable,
see \cite[Proposition~1.6]{MR2867756}.
We also assume that $\# X\ge 2$.

The diameter of a set $A\subset X$ is denoted by $\diam(A)$. 
We write $n_0=-\infty$ if $\diam(X)=\infty$, and otherwise choose $n_0$ 
to be the smallest integer
such that $2^{-n_0}\le 2\diam(X)$. In the latter
case, we have $\diam(X)<2^{-n_0}\le 2\diam(X)$. 

We also assume throughout the paper
that the measure $\mu$ is {\em doubling}, that is,
there is a constant $c_\mu> 1$, called
the {\em doubling constant of $\mu$}, such that
\begin{equation}\label{e.doubling}
\mu(2B) \le c_\mu\, \mu(B)
\end{equation}
for all balls $B=B(x,r)$ in $X$. 
Here we use for $0<t<\infty$ the notation $tB=B(x,tr)$. 

If $X$ is connected,  
then the doubling measure $\mu$ also satisfies
the {\em reverse doubling condition}, that is,
there is a constant  $0<c_R=C(c_\mu)<1$ such that
\begin{equation}\label{e.rev_dbl_decay}
\mu(B(x,r/2))\le c_R\, \mu(B(x,r))
\end{equation}
for all $x\in X$ and all $0<r<\diam(X)/2$;
see for instance~\cite[Lemma~3.7]{MR2867756}.
Iteration of~\eqref{e.rev_dbl_decay} shows that
if $X$ is connected, then there exist
an exponent $\sigma>0$ and a constant $c_\sigma>0$,
both depending on $c_\mu$ only, such that the {\em quantitative 
reverse doubling condition}
\begin{equation}\label{e.reverse_doubling}
 \frac{\mu(B(x,r))}{\mu(B(x,R))} \le c_\sigma\Bigl(\frac{r}{R}\Bigr)^\sigma
\end{equation}
holds for all $x\in X$ and all $0<r<R\le 2\diam(X)$.
However, in most of our results we will not assume $X$ to be connected,
but instead we often explicitly assume that \eqref{e.reverse_doubling}
holds for some exponent $\sigma>0$.

We denote the closure of a set $A\subset X$ by $\iol{A}$. In particular, if $B=B(x,r)\subset X$ is a ball,
then the notation $\iol{B}=\iol{B(x,r)}$ refers to the closure of the ball $B$.  
The characteristic function of a set $A\subset X$ is denoted by $\mathbf{1}_{A}$; 
that is, $\mathbf{1}_{A}(x)=1$ if $x\in A$
and $\mathbf{1}_{A}(x)=0$ if $x\in X\setminus A$.

When $A\subset X$ is a measurable set 
with $0<\mu(A)<\infty$,
we write
\[
f_A=\vint_{A} f(y)\,d\mu(y)=\frac{1}{\mu(A)}\int_A f(y)\,d\mu(y)
\]
for the integral average of $f\in L^1(A)$ over $A$.
 If $f:X\to \R$ is a measurable function, then the {\em noncentered Hardy--Littlewood
 maximal function} $M^*f$ of $f$ is defined by
\begin{equation}\label{e.max_funct_def_noncentered}
M^*f(x)=\sup_{B}\vint_{B} \lvert f(y)\rvert\,d\mu(y)\,,\qquad x\in X\,,
\end{equation}
where the supremum is taken over all balls $B\subset X$ such that $x\in B$. The sublinear operator $M^*$ is bounded on $L^p(X)$ 
with a constant $C(c_\mu,p)$
for every $1<p\leq \infty$, see \cite[Theorem 3.13]{MR2867756}.  

Finally, we say that a function $\eta\colon X\to\R$ is $L$-Lipschitz, if
$\lvert \eta(x)-\eta(y)\rvert \le Ld(x,y)$ for every $x,y\in X$.

\section{Fractional Haj{\l}asz gradient}\label{s.frac_haj}

We define fractional Haj{\l}asz gradients in open sets as in \cite{MR3471303}, by adapting~\cite[Definition~1.1]{MR2764899}.

\begin{definition}\label{d.fhg}
Let $0<\beta\le 1$ and let $u\colon \Omega\to\R$ be a measurable function on an open set $\Omega\subset X$. A sequence $g=(g_k)_{k\in\Z}$ of nonnegative  measurable functions is called a fractional $\beta$-Haj{\l}asz gradient of $u$, if there exists
a set $E\subset \Omega$ such that $\mu(E)=0$ and for all 
$k\in \Z$ and $x,y\in \Omega\setminus E$ satisfying $2^{-k-1}\le d(x,y)<2^{-k}$, we have
\[
\lvert u(x)-u(y)\rvert \le d(x,y)^\beta(g_k(x)+g_k(y))\,.
\]
We denote by $\mathbb{D}^\beta_\Omega(u)$ the collection of all fractional
$\beta$-Haj{\l}asz gradients of $u$.
\end{definition}

We need the following
metric space variant of the Poincar\'e type inequality \cite[Lemma~2.1]{MR2764899}.

\begin{lemma}\label{lemma: Poincare}
Let $0<\beta<1$ and assume that $\mu$ satisfies the  
reverse doubling condition \eqref{e.rev_dbl_decay} with a constant $0<c_R<1$.
Let $x_0\in X$, let $n\in\Z$ be such that $2^{-n+3}<\diam(X)$, and assume
that $u\in L^1(B(x_0,2^{-n+2}))$ and 
$(g_k)_{k\in\Z}\in \mathbb{D}_{B(x_0,2^{-n+2})}^\beta(u)$.
Then
\[
\vint_{B(x_0,2^{-n})}|u(x)-u_{B(x_0,2^{-n})}|\,d\mu(x)
\le C(c_\mu,c_R)2^{-n\beta}\sum_{k= n-3}^n\;\vint_{B(x_0,2^{-n+2})}g_k(x)\,d\mu(x)\,.
\]
\end{lemma}

\begin{proof}
We estimate
\begin{equation}\label{e.poi1}
\begin{split}
&\vint_{B(x_0,2^{-n})}|u(x)-u_{B(x_0,2^{-n})}|\,d\mu(x)
\le 2\vint_{B(x_0,2^{-n})}|u(x)-u_{B(x_0,2^{-n+2})\setminus B(x_0,2^{-n+1})}|\,d\mu(x)\\
&\quad \le 2\vint_{B(x_0,2^{-n})}\vint_{B(x_0,2^{-n+2})\setminus B(x_0,2^{-n+1})}|u(x)-u(y)|\,d\mu(y)\,d\mu(x)\,.
\end{split}
\end{equation}
We let $E\subset B(x_0,2^{-n+2})$ be the exceptional set for $(g_k)_{k\in\Z}$ as in Definition \ref{d.fhg}.  
Fix points $x\in B(x_0,2^{-n})$ and $y\in B(x_0,2^{-n+2})\setminus B(x_0,2^{-n+1})$ such that $x,y\not\in E$. Then
\[
2^{-n}\le d(x,y) < 2^{-n+3}
\]
and therefore
\[
\lvert  u(x)-u(y)\rvert \le d(x,y)^\beta\left(\sum_{k=n-3}^{n-1} g_k(x) +\sum_{k=n-3}^{n-1} g_k(y)\right)\,.
\]
Since $\mu(E)=0$, it follows that
\begin{equation}\label{e.poi2}
\begin{split}
&\vint_{B(x_0,2^{-n})}\vint_{B(x_0,2^{-n+2})\setminus B(x_0,2^{-n+1})}|u(x)-u(y)|\,d\mu(y)\,d\mu(x)\\&
\quad \le 
2^{-n\beta+3}\biggl(\vint_{B(x_0,2^{-n})} \sum_{k=n-3}^{n-1} g_k(x)\,d\mu(x) + \vint_{B(x_0,2^{-n+2})\setminus B(x_0,2^{-n+1})}\sum_{k=n-3}^{n-1} g_k(y)\,d\mu(y)\biggr)\,. 
\end{split}
\end{equation}
By the doubling condition \eqref{e.doubling}, we get
\begin{equation}\label{e.poi3}
\vint_{B(x_0,2^{-n})} \sum_{k=n-3}^{n-1} g_k(x)\,d\mu(x)\le 
c_\mu^2\sum_{k=n-3}^{n-1} \vint_{B(x_0,2^{-n+2})} g_k(x)\,d\mu(x)\,.
\end{equation}
By the assumed reverse doubling condition \eqref{e.rev_dbl_decay}, we have
\[
\mu(B(x_0,2^{-n+1}))\le c_R\,\mu(B(x_0,2^{-n+2}))\,,
\]
and therefore
\begin{align*}
\mu(B(x_0,2^{-n+2})\setminus B(x_0,2^{-n+1}))\ge (1-c_R)\mu(B(x_0,2^{-n+2}))\,.
\end{align*}
Hence,
\begin{equation}\label{e.poi4}
\vint_{B(x_0,2^{-n+2})\setminus B(x_0,2^{-n+1}))}\sum_{k=n-3}^{n-1} g_k(y)\,d\mu(y)
\le \frac{1}{1-c_R}\sum_{k=n-3}^{n-1} \vint_{B(x_0,2^{-n+2})} g_k(x)\,d\mu(x)\,.
\end{equation}
The desired Poincar\'e inequality follows by combining inequalities \eqref{e.poi1}, \eqref{e.poi2}, \eqref{e.poi3} and
\eqref{e.poi4}.
\end{proof}

The following Leibniz type rule for $\beta$-Haj{\l}asz gradients is from \cite[Lemma 3.10]{MR3471303}. 
The proof is a case study.

\begin{lemma}\label{LemWithLipForTL}
Let $0<\beta<1$ and
let $(g_k)_{k\in\Z}\in\mathbb{D}^\beta_X(u)$ be a fractional $\beta$-Haj{\l}asz gradient of a measurable function $u\colon X\to \R$.
Assume that $\eta$ is a bounded $L$-Lipschitz function in $X$
and define
\[
\rho_k=\big(g_k\Vert \eta\Vert_{\infty}+2^{k(\beta-1)}L|u|\big)\mathbf{1}_{\{x\in X\,:\,\eta(x)\not=0\}}\,,\qquad k\in\Z\,.
\]
Then $(\rho_k)_{k\in\Z}\in\mathbb{D}^\beta_X(\eta u)$.
\end{lemma}

\section{Triebel--Lizorkin spaces and relative capacity}\label{s.tl_space}

Haj{\l}asz--Triebel--Lizorkin spaces were introduced in \cite{MR2764899}. The associated global capacities 
first appeared in~\cite{MR3605979} and 
they have been  studied, for instance, in~\cite{MR4104350}, \cite{KM} and \cite{math9212724}.
See also~\cite{MR1411441} and the references therein for 
the theory of Triebel--Lizorkin capacities in the Euclidean case.
In this section, we first recall Haj{\l}asz--Triebel--Lizorkin spaces
on open sets and then we define the corresponding relative capacities that seem to be novel in this generality.

Fix $1\le p<\infty$, $1\le q\le \infty$, and an open set $\Omega\subset X$. For a sequence $g=(g_k)_{k\in\Z}$ of measurable
functions in $\Omega$, we define the mixed norm 
\[
\lVert g \rVert_{L^p(\Omega;\,l^q)}
=\bigl\|\|g\|_{l^q(\Z)}\bigr\|_{L^p(\Omega)}\,,
\]
where
\[
\|g\|_{l^q(\Z)}=
\begin{cases}
\Bigl(\sum_{k\in\Z}|g_{k}|^{q}\Bigr)^{1/q},&\quad\text{ for }1\le q<\infty, \\
\;\sup_{k\in\Z}|g_{k}|,&\quad\text{ for }q=\infty.
\end{cases}
\]
We use the following
Fefferman--Stein vector-valued maximal function inequality; we refer to~\cite[Theorem~1.2]{MR2542655} for a proof.
If $1<p<\infty$ and $1< q\le \infty$, then there
exists a positive constant $C>0$ such that for all sequences of measurable
functions $(g_k)_{k\in\Z}$ in $X$ we have
\begin{equation}\label{e.fs}
\big\lVert (M^*g_k)_{k\in\Z}\big\rVert_{L^p(X;l^q)}\le C\lVert (g_k)_{k\in\Z}\rVert_{L^p(X;l^q)}\,,
\end{equation}
where $M^* g_k$ is the noncentered maximal function of $g_k$ for each $k\in\Z$.

\begin{definition}
Let $1<p<\infty$, $1\le q\le \infty$, and $0<\beta < 1$,
and let $\Omega\subset X$ be an open set. 
The homogeneous Haj{\l}asz--Triebel--Lizorkin space $\dot{M}^\beta_{p,q}(\Omega)$ 
is the seminormed space of all measurable functions $u\colon\Omega\to\R$ satisfying
\[
\lvert u\rvert_{\dot{M}^{\beta}_{p,q}(\Omega)}=\inf_{g\in\mathbb{D}^\beta_\Omega(u)} \lVert g\rVert_{L^p(\Omega;l^q)}<\infty\,.
\]
\end{definition}

Notice that $\lvert u\rvert_{\dot{M}^{\beta}_{p,q}(\Omega)}\le \lvert u\rvert_{\dot{M}^{\beta}_{p,\tau}(\Omega)}$ 
whenever $1\le \tau\le q\le \infty$.
We remark that $\lvert \,\cdot\, \rvert_{\dot{M}^{\beta}_{p,q}(\Omega)}$ is a seminorm
but not a norm since $\lvert u\rvert_{\dot{M}^{\beta}_{p,q}(\Omega)}=0$
whenever $u$ is a constant function. 

The following definition of a relative Haj{\l}asz--Triebel--Lizorkin $(\beta,p,q)$-capacity $\capm{\beta}{p}{q}$
is motivated by~\cite[Definition 4.1]{MR4649157}; see also \cite[Definition~7.1]{MR3605166}, \cite{MR3331699} and 
\cite[\S 11]{Mazya2011}.

\begin{definition}\label{d.capacity}
Let $1< p<\infty$, $1\le q\le \infty$, $0<\beta < 1$, and $\Lambda\ge  2$. 
Let $B\subset X$ be a ball and let $E\subset \iol{B}$ be a closed set. Then we write
\[
\capm{\beta}{p}{q} (E,2B,\Lambda B) = \inf_\varphi \vert \varphi\vert_{\dot{M}^{\beta}_{p,q}(\Lambda B)}^p\,,
\]
where the infimum is taken over all continuous functions 
$\varphi\colon X\to \R$  such that $\varphi(x)\ge 1$ for every $x\in E$ and 
$\varphi(x)=0$ for every $x\in X\setminus 2B$.
\end{definition}

Observe that if $F\subset E\subset\iol{B}$ are closed sets and $1\le \tau\le q\le \infty$,
then 
\begin{equation}\label{e.cap_Eq_comp}
\capm{\beta}{p}{q} (F,2B,\Lambda B)\le \capm{\beta}{p}{q} (E,2B,\Lambda B)\le \capm{\beta}{p}{\tau} (E,2B,\Lambda B)\,.
\end{equation}
Next we estimate the relative capacities of balls.
In part (a) we have $q=1$ and in part (b) $q=\infty$; compare to \eqref{e.cap_Eq_comp}.

\begin{theorem}\label{t.cap_balls}
Let $1<p<\infty$, $0<\beta <1$, $x_0\in X$, and $r>0$.
\begin{itemize}
\item[(a)]
If $\Lambda \ge 2$, %
then
\[
\capm{\beta}{p}{1} \bigl(\overline{B(x_0,r)},B(x_0,2r),B(x_0,\Lambda r)\bigr)\le  
C(c_\mu,\beta,p,\Lambda) r^{-\beta p}\mu(B(x_0,r))\,.
\]
\item[(b)]
Assume in addition that $X$ is connected
and $r<(1/8)\diam(X)$.
If $\Lambda> 2$, 
then
\[
r^{-\beta p}\mu(B(x_0,r))\le 
C(\beta,p, c_\mu,\Lambda)\capm{\beta}{p}{\infty} \bigl(\overline{B(x_0,r)},B(x_0,2r),B(x_0,\Lambda r)\bigr)\,.
\]
\end{itemize}
\end{theorem}

\begin{proof}
(a) Write $B=B(x_0,r)$ and let
\[
\varphi(x)=\max \bigl\{0,1-r^{-1}\dist(x,B)\bigr\}\,
\]
for every $x\in X$. 
Then $0\le\varphi\le 1$ in $X$, $\varphi=1$ in $\iol{B}$, $\varphi=0$ in $X\setminus 2B$, 
and  $\varphi$ is an $r^{-1}$-Lipschitz function in $X$.
Since $\varphi$ is admissible for the relative capacity, we have
\[
\capm{\beta}{p}{1} (\iol{B},2B,\Lambda B) \le \lvert \varphi\rvert_{\dot{M}^{\beta}_{p,1}(\Lambda B)}^p\,.
\]
Lemma~\ref{LemWithLipForTL}, with $u=1$ and $g_k=0$ for all $k\in\Z$, implies that
$(\rho_k)_{k\in\Z}\in\mathbb{D}^\beta_X(\varphi)$, where
\[
\rho_k=2^{k(\beta-1)}r^{-1}\mathbf{1}_{\{x\in X\,:\,\varphi(x)\not=0\}}\,,\qquad \text{ for all }k\in\Z\,.
\]
Let $m$ be the smallest integer
such that $2^{-m-1}\le 2\Lambda r$.
Observe that $(\mathbf{1}_{k\ge m}\rho_k)_{k\in\Z}\in  \mathbb{D}_{\Lambda B}^\beta(\varphi)$.
Thus
\begin{align*}
\lvert \varphi\rvert_{\dot{M}^{\beta}_{p,1}(\Lambda B)}^p
&\le \big\lVert (\mathbf{1}_{k\ge m} \rho_k)_{k\in\Z}\big\rVert_{L^p(\Lambda B;l^1)}^p
=\int_{\Lambda B} \left(\sum_{k=m}^\infty \rho_k(x)\right)^{p}\,d\mu(x)
\\&\le r^{-p}\int_{2 B} \left(\sum_{k=m}^\infty 2^{k(\beta-1)}\right)^{p}\,d\mu(x)\\
&\le C(\beta,p) r^{-p} 2^{mp(\beta-1)} \mu(2B)\le C(c_\mu,\beta,p,\Lambda) r^{-\beta p}\mu(B)\,,
\end{align*}
and the claim (a) follows from this.

(b) The proof is similar to the proof of Lemma~\ref{lemma: Poincare}.
Write $B=B(x_0,r)$ and choose $\lambda=\min\{3,\Lambda\}>2$.
Let $\varphi\colon X\to \R$ be a continuous function such that $\varphi(x)\ge 1$ for every $x\in \iol{B}$ and 
$\varphi(x)=0$ for every $x\in X\setminus 2B$,
and let $g=(g_k)_{k\in\Z}\in\mathbb{D}^\beta_{\Lambda B}(\varphi)$.

Fix $m\in\Z$ such that $2^{-m-1}\le r<2^{-m}$. Then
\[
2^{-m-1}\le r \le  d(x,y) <(1+\lambda)r\le  4r<2^{-m+2}
\]
for all $x\in B$ and $y\in \lambda B\setminus 2 B$.
Since $\varphi\ge 1$ in $B$ and $\varphi=0$ in $\lambda B\setminus 2B$,
we can proceed as in \eqref{e.poi1} and obtain
\begin{equation}\label{e.spoi1}
\begin{split}
1&\le \vint_{B} \varphi(x)\,d\mu(x)
=\vint_{B} \lvert \varphi(x)-\varphi_{\lambda B\setminus 2B}\rvert\,d\mu(x)\\
&\le \vint_{B}\vint_{\lambda B\setminus 2B} \lvert \varphi(x)-\varphi(y)\rvert\,d\mu(y)\,d\mu(x)\\ 
&\le 4r^\beta\sum_{k=m-2}^{m}\vint_{B}  g_k(x)\,d\mu(x)+4r^\beta\sum_{k=m-2}^{m}\vint_{\lambda B\setminus 2B} g_k(y)\,d\mu(y)\,.
\end{split}
\end{equation}
Here
\begin{equation}\label{e.spoi3}
\begin{split}
\sum_{k=m-2}^{m}\vint_{B}  g_k(x)\,d\mu(x) & \le \sum_{k=m-2}^{m}\left(\vint_{B}  g_k(x)^p\,d\mu(x)\right)^{1/p}\\
& \le 3\left(\vint_{B}  \left(\sup_{k\in\Z} g_k(y)\right)^{p}\,d\mu(x)\right)^{1/p}\,.
\end{split}
\end{equation}

Since $X$ is connected and $\lambda r<4r<\diam(X)/2$, by~\cite[Lemma~3.7]{MR2867756} there exists a constant $0<c_R=C(c_\mu,\Lambda)<1$ such that
\[
\mu(2B)=\mu(B(x_0,2r))\le c_R\,\mu(B(x_0,\lambda r))=c_R\,\mu(\lambda B)\,,
\]
and therefore
\begin{align*}
\mu(\lambda B\setminus 2B)\ge (1-c_R)\mu(\lambda B)\,.
\end{align*}
Hence,
\begin{equation}\label{e.spoi4}
\begin{split}
\sum_{k=m-2}^{m} \vint_{\lambda B\setminus 2B}g_k(y)\,d\mu(y)
&\le \frac{1}{1-c_R}\sum_{k=m-2}^{m}  \vint_{\lambda B} g_k(x)\,d\mu(x)\\
&\le \frac{1}{1-c_R}\sum_{k=m-2}^{m}  \left(\vint_{\lambda B} g_k(x)^p\,d\mu(x)\right)^{1/p}\\
&\le \frac{3}{1-c_R}  \left(\vint_{\lambda B} \left(\sup_{k\in\Z} g_k(y)\right)^{p}\,d\mu(x)\right)^{1/p}\,.
\end{split}
\end{equation}
By combining inequalities \eqref{e.spoi1},  \eqref{e.spoi3} and \eqref{e.spoi4}, we obtain
\begin{align*}
1
&\le \left(\vint_{B}\vint_{\lambda B\setminus 2B} \lvert \varphi(x)-\varphi(y)\rvert\,d\mu(y)\,d\mu(x)\right)^p\\
&\le C(p,c_R)r^{\beta p}\left(\vint_{B}  \left(\sup_{k\in\Z} g_k(y)\right)^{p}\,d\mu(x)+\vint_{\lambda B} \left(\sup_{k\in\Z} g_k(y)\right)^{p}\,d\mu(x)\right)\,,
\end{align*}
and thus
\[
r^{-\beta p}\mu(B)\le C(p,c_R,\beta)\int_{\Lambda  B} \left(\sup_{k\in\Z} g_k(y)\right)^{p}\,d\mu(x)\,.
\]
Since $c_R=C(c_\mu,\Lambda)$, the desired inequality follows by taking infimum over all functions
$\varphi$ and $g$ as above.
\end{proof}

Observe the strict lower bound $\Lambda > 2$ in part (b) of Theorem \ref{t.cap_balls}.
The following Example~\ref{e.ball_example},
together with inequality~\eqref{e.cap_Eq_comp}, shows that we cannot relax this lower bound to 
$\Lambda \ge  2$.

\begin{example}\label{e.ball_example}
Let $X=\R^n$ equipped with the Euclidean distance and the Lebesgue measure.
Fix $1<p<\infty$ and $0<\beta<1$ such that $\beta p<1$.
We show that 
\begin{equation}\label{e.zero_cap}
\capm{\beta}{p}{1}\bigl(\overline{B(0,1)},B(0,2),B(0,2)\bigr)=0\,.
\end{equation}

Fix $j\in\N$ and
let
\[
\varphi_j(x)=\min \Bigl\{1, 2^j\dist(x,\R^n\setminus B(0,2))\Bigr\}
\] 
for every $x\in \R^n$.
Observe that $0\le \varphi_j\le 1$ in $\R^n$, $\varphi_j(x)=1$ for every $x\in \overline{B(0,1)}$ and $\varphi_j(x)=0$ for
every $\R^n\setminus B(0,2)$. It follows that $\varphi_j$ is an admissible
test function for the capacity on the left-hand side of \eqref{e.zero_cap}.
Write $B_j=B(0,2-2^{-j})$ and $A_j=B(0,2)\setminus B_j$.
For every $k\in\Z$, we define
\[
g_k(x)=\begin{cases}
0\,,\quad &\text{ if }k<-2\,,\\
2^{j\beta}\mathbf{1}_{A_j}\,,\quad &\text{ if }-2\le k< j\,,\\
2^{j-k(1-\beta)}\mathbf{1}_{A_j}\,,\quad &\text{ if }j\le k\,.
\end{cases}
\]
Since $\varphi_j=1$ in $B_j$, a straightforward case study shows that $(g_k)_{k\in\mathbb{Z}}\in \mathbb{D}^\beta_{B(0,2)}(\varphi_j)$. 
Hence, 
\begin{align*}
&\capm{\beta}{p}{1}\bigl(\overline{B(0,1)},B(0,2),B(0,2)\bigr)
\le  \lvert \varphi_j\rvert_{\dot{M}^{\beta}_{p,1}(B(0,2))}^p\le \int_{B(0,2)} \left(\sum_{k\in\Z} g_k(x)\right)^{p}\,d\mu(x)\\
&\qquad \le 2^{p-1}\int_{A_j} \left(\sum_{k=-2}^{j-1} 2^{j\beta}\right)^{p}\,d\mu(x)+2^{p-1}\int_{A_j} \left(\sum_{k=j}^\infty 2^{j-k(1-\beta)}\right)^{p}\,d\mu(x)\,.
\end{align*}
Here
\begin{align*}
\int_{A_j} \left(\sum_{k=-2}^{j-1} 2^{j\beta}\right)^{p}\,d\mu(x)
\le \lvert A_j\rvert(j+2)^{p}2^{j\beta p}\le C(n)2^{-j(1-\beta p)}(j+2)^{p}
\end{align*}
and
\begin{align*}
\int_{A_j} \left(\sum_{k=j}^\infty 2^{j-k(1-\beta)}\right)^{p}\,d\mu(x)
& = \lvert A_j\rvert 2^{j\beta p}\left(\sum_{k=j}^\infty 2^{(j-k)(1-\beta)}\right)^{p}\\
& \le C(p,n,\beta)2^{-j(1-\beta p)}\,.
\end{align*}
Since $\beta p<1$, by taking $j\to\infty$ we see that  \eqref{e.zero_cap} holds.
\end{example}

We need
suitable Hausdorff contents to obtain more elaborate lower
bounds for relative capacities of general sets,
see Theorem \ref{t.HC_bp} below.

\begin{definition}\label{d.hcc}
Let $0<\rho\le\infty$ and $d\ge 0$.
The ($\rho$-restricted) Hausdorff content of codimension $d$ 
of a set $F\subset X$ is defined by  
\[
\Ha^{\mu,d}_\rho(F)=\inf\Biggl\{\sum_{k} \mu(B(x_k,r_k))\,r_k^{-d} \,:\,
F\subset\bigcup_{k} B(x_k,r_k)\text{ and } 0<r_k\leq \rho  \Biggr\}\,.
\]
\end{definition}

For $q=\infty$ we have the following lower bound estimate for relative capacities
in terms of Hausdorff contents. Observe that by \eqref{e.cap_Eq_comp} the
corresponding result holds for all $1\le q <\infty$ as well. 
We refer to~\cite{KM} for recent advances on related  lower bound estimates.

\begin{theorem}\label{t.HC_bp}
Let $1<p<\infty$, $0\le \eta<p$, $0<\beta <1$, and $\Lambda\ge 41$, and
assume that $\mu$ satisfies the
reverse doubling condition \eqref{e.rev_dbl_decay} with a constant $0<c_R<1$. 
Let $B=B(x_0,r)$, where $x_0\in X$ and $0<r<(1/80)\diam(X)$,
and let $E\subset \iol{B}$ be a closed set.
Then there exists a constant $C=C(\beta,p,\eta,c_R,c_\mu,\Lambda)>0$ such that
\[
\mathcal{H}^{\mu,\beta\eta}_{5\Lambda r}(E) \le 
C r^{\beta(p-\eta)}\capm{\beta}{p}{\infty}(E,2B,\Lambda B).
\]
\end{theorem}

\begin{proof}
Let $\varphi\colon X\to \R$ be a continuous function   such that $\varphi(x)\ge 1$ for every $x\in E$ and 
$\varphi(x)=0$ for every $x\in X\setminus 2B$
and let $g=(g_k)_{k\in\Z}\in\mathbb{D}^\beta_{\Lambda B}(\varphi)$.
By replacing the function $\varphi$ with $\max\{0,\min\{\varphi,1\}\}$, if necessary, we may assume that $0\le \varphi \le 1$. 
Thus $\varphi$ is continuous on $X$,  $\varphi=1$ on $E$, and $\varphi=0$ on $X\setminus 2B$. 
Fix $m\in \Z$ such that
$2^{-m-1}\le 5r<2^{-m}$. Let $x\in E$ and write $B_0=B(x,2^{-m})$.

Since $8r<80r<\diam(X)$, by inequality \eqref{e.rev_dbl_decay} we have
\begin{align*}
0\le \varphi_{B_0}=\vint_{B_0} \varphi(y)\,dy \le \frac{\mu(2B)}{\mu(B_0)}
\le \frac{\mu(B(x_0,2r))}{\mu(B(x_0,4r))}
\le c_R < 1.
\end{align*}
As a consequence, since $x\in E$, we find that
\[
\lvert \varphi(x)- \varphi_{B_0}\rvert  \ge 1-c_R>0.
\]
Define $B_n=2^{-n}B_0=B(x,2^{-(m+n)})$ for each $n\in\N$.
By the doubling property of $\mu$,
\begin{align*}
0<1-c_R & \le \lvert \varphi(x)-\varphi_{B_0}\rvert=\lim_{k\to \infty} \lvert \varphi_{B_k}-\varphi_{B_0}\rvert=\lim_{k\to\infty} \left\lvert \sum_{n=0}^{k-1} (\varphi_{B_{n+1}}-\varphi_{B_n})\right\rvert\\
&\le \sum_{n=0}^\infty \intav_{B_{n+1}} \lvert \varphi(y)-\varphi_{B_n}\rvert\, d\mu(y)
\le c_\mu\sum_{n=0}^\infty \intav_{B_{n}} \lvert \varphi(y)-\varphi_{B_n}\rvert\, d\mu(y)\,.
\end{align*}
Fix $n\in\N_0$.
Observe that $2^{-(m+n)+3}\le 80r<\diam(X)$ and that
$B(x,2^{-(m+n)+2})\subset \Lambda B$.
  Hence, Lemma~\ref{lemma: Poincare} implies that
\begin{align*}
\intav_{B_{n}} \lvert \varphi(y)-\varphi_{B_n}\rvert\, d\mu(y)&\le 
C(c_\mu,c_R)2^{-(m+n)\beta}\sum_{k= n+m-3}^{n+m}\;\vint_{B(x,2^{-(m+n)+2})}g_k(y)\,d\mu(y)\,.
\end{align*}
By combining the above estimates, we get
\begin{align*}
0 & < 1-c_R\le c_\mu\sum_{n=0}^\infty \intav_{B_{n}} \lvert \varphi(y)-\varphi_{B_n}\rvert\, d\mu(y)\\
&\le 
C(c_\mu,c_R)\sum_{n=0}^\infty2^{-(m+n)\beta} \sum_{k= n+m-3}^{n+m}\vint_{B(x,2^{-(m+n)+2})} g_k(y)\,d\mu(y)\\
&\le 
C(c_\mu,c_R)\sum_{n=m-2}^\infty 2^{-n \beta}\;\sum_{k= n-1}^{n+2}\vint_{B(x,2^{-n})}  g_k(y)\,d\mu(y)\\
&\le  
C(c_\mu,c_R)\sum_{n=m-2}^\infty 2^{-n \beta}\;\sum_{k= n-1}^{n+2}\left(\vint_{B(x,2^{-n})}   g_k(y)^p\,d\mu(y)\right)^{1/p}\\
&\le  
C(c_\mu,c_R)\sum_{n=m-2}^\infty 2^{-n \beta}\;\left(\vint_{B(x,2^{-n})}  \left(\sup_{k\in\Z} g_k(y)\right)^{p}\,d\mu(y)\right)^{1/p}\,.
\end{align*}

Let $\delta=\beta(p-\eta)/p>0$. Then
\begin{align*}
&2^{(m-2)\delta}\sum_{n=m-2}^\infty 2^{-n\delta}=C(\beta,p,\eta,c_R)(1-c_R)
\\&\qquad \le C(\beta,p,\eta,c_R,c_\mu)\sum_{n=m-2}^\infty 2^{-n \beta}\;\left(\vint_{B(x,2^{-n})}   \left(\sup_{k\in\Z} g_k(y)\right)^{p}\,d\mu(y)\right)^{1/p}\,.
\end{align*}
In particular, by comparing the above sums it follows that there exists $n\ge m-2$,
depending on $x$, such that
\[
\mu(B(x,2^{-n})) (2^{-n})^{-\beta \eta}\le C(\beta,p,\eta,c_R,c_\mu)2^{-\beta m(p-\eta)}\;\int_{B(x,2^{-n})}   \left(\sup_{k\in\Z} g_k(y)\right)^{p}\,d\mu(y).
\]
Write $r_x=2^{-n}$ and
$B_x=B(x,r_x)=B(x,2^{-n})$. Then the previous estimate gives
\[
\mu(B_x) r_x^{-\beta \eta}\le C(\beta,p,\eta,c_R,c_\mu)r^{\beta (p-\eta)}\;\int_{B_x}   \left(\sup_{k\in\Z} g_k(y)\right)^{p}\,d\mu(y).
\]

By the $5r$-covering lemma \cite[Lemma~1.7]{MR2867756}, we obtain points $x_\ell\in E$,
$\ell\in\N$, such that the balls $B_{x_\ell}\subset \Lambda B$ 
with radii $r_{x_\ell}\le \Lambda r$ are
pairwise disjoint and 
$E\subset \bigcup_{\ell=1}^\infty 5B_{x_\ell}$. Hence,
\begin{align*}
\mathcal{H}^{\mu,\beta\eta}_{5\Lambda r}(E) &\le \sum_{\ell=1}^\infty  \mu(5B_{x_\ell})
(5r_{x_\ell})^{-\beta \eta}%
 \le  Cr^{\beta (p-\eta)}\;\sum_{\ell=1}^\infty\int_{B_{x_\ell}}   \left(\sup_{k\in\Z} g_k(y)\right)^{p}\,d\mu(y)\\
 &\le 
    Cr^{\beta (p-\eta)}\;\int_{\Lambda B}  \left(\sup_{k\in\Z} g_k(y)\right)^{p}\,d\mu(y)\,.
 \end{align*}
The desired inequality follows by taking infimum over all functions
$g$ and $\varphi$ as above.
\end{proof}

\section{Discrete potential operator}\label{s.discr_pot}

In this section we introduce and study a discrete potential operator $\Lop_\beta$.
The main importance of this operator is that it majorizes the potential $\Hop_\beta$
that is used in the next section to define the Triebel--Lizorkin $(\beta,p,q)$-capacity $\Cp^{\beta}_{p,q}$;
see Definitions~\ref{d.hbeta} and~\ref{d.cpq}. 
Moreover, due to the Lipschitz-continuity of the bump functions $\psi_{n,i}$
that are used in the definition of $\Lop_\beta$, this operator has better 
smoothing properties than $\Hop_\beta$ whose definition is based on 
characteristic functions of  balls. In particular, 
Lemma~\ref{l.pot_norm_est} seem not to be available for $\Hop_\beta$.

We begin with a construction of a partition of unity, following the
ideas from~\cite{MR1954868} and \cite{MR2747071}.
Let $n\in\Z$ be such that $n\ge n_0$;
recall that $2^{-n_0}$ is essentially $\diam(X)$. 
By a maximal packing argument~\cite[p.~101]{MR1800917} 
and using the separability of $X$, we may select a %
countable number of points $x_{n,i}\in X$, $i\in P_n$, 
such that the balls $B(x_{n,i},2^{-n-1})$, $i\in P_n$, are pairwise disjoint, while
\begin{equation}\label{e.coverX}
X = \bigcup_{i\in P_n} B(x_{n,i},2^{-n}).
\end{equation}
Since $X$ is equipped with a doubling measure $\mu$, the
balls $B(x_{n,i}, 6\cdot 2^{-n})$ have a finite overlapping property, 
that is, there exists a~constant $N=C(c_\mu)$ such that
\begin{equation}\label{e.overlap}
\sum_{i\in P_n} \mathbf{1}_{B(x_{n,i}, 6\cdot 2^{-n})} \leq N.
\end{equation}
In addition, there are constants $\kappa=C(c_\mu)$ and $L=C(c_\mu)$ and
a family of functions $\psi_{n,i}$, $i\in P_n$, 
with the following properties: 
$0\le \psi_{n,i}\le 1$, 
$\psi_{n,i}=0$ in $X\setminus B(x_{n,i},6\cdot 2^{-n})$, 
$\psi_{n,i}\ge \kappa$ in $B(x_{n,i}, 3\cdot 2^{-n})$, $\psi_{n,i}$ is $2^nL$-Lipschitz in $X$, and
\[
\sum_{i\in P_n} \psi_{n,i}(x)=1
\]
for every $x\in X$.

The family $\psi_{n,i}$, $i\in P_n$, is called a partition of unity subordinate to the cover $B(x_{n,i},2^{-n})$.
We keep these families, for $n\ge n_0$, fixed throughout the rest of the paper, starting from the
following definition.

\begin{definition}\label{d.dpotential}
Let $\beta >0$ and let $f=(f_n)_{n=n_0}^\infty$ be a sequence of nonnegative Borel
functions in $X$. 
The discrete potential $\Lop_\beta f$ is defined by
\[
\Lop_\beta f(x)=\sum_{n=n_0}^{\infty} 2^{-\beta n} \sum_{i\in P_n}\psi_{n,i}(x)\vint_{B(x_{n,i},3\cdot 2^{-n})} f_n(y)\,d\mu(y)\,,\qquad x\in X\,.
\]
\end{definition}

The closely related discrete maximal operator is studied for instance in~\cite{MR2747071,MR3102566}.
We also mention that a median type approach to discrete potential operators appears in~\cite{MR3605979}, 
where approximation properties of 
Haj{\l}asz--Triebel--Lizorkin functions are studied in the full range $0<p,q<\infty$.

The following lemma allows us to control the Haj{\l}asz--Triebel--Lizorkin seminorm
of the discrete potential $\Lop_\beta f$ in terms of a mixed norm of $f$.

\begin{lemma}\label{l.pot_norm_est}
Let $1<p<\infty$, $1<q\le\infty$, and $0<\beta <1$, 
and assume that $f=(f_n)_{n=n_0}^\infty$ is a sequence of nonnegative Borel functions in $X$. 
Then there exists a constant $C=C(c_\mu,\beta,p,q)$  such that
\[
\lvert \Lop_\beta f\rvert_{\dot{M}^{\beta}_{p,q}(X)}\le C\lVert f\rVert_{L^p(X;l^q)}\,.
\]
\end{lemma}

\begin{proof}
We may clearly assume that $\lVert f\rVert_{L^p(X;l^q)}<\infty$.
If $X$ is bounded, then we extend $f$ as a sequence $(f_n)_{n\in\Z}$ by
setting $f_n=0$ if $n<n_0$. 
Define $g=(g_k)_{k\in \Z}$, where
\[
g_k(x)=Nc_\mu 2^\beta \sum_{n=k}^\infty 2^{\beta (k-n)}M^*f_n(x) +  LNc_\mu \sum_{n=-\infty}^{k-1} 2^{(\beta-1)(k-n)}
M^*f_n(x)
\]
for every $x\in X$ and $k\in\Z$, 
and $N=C(c_\mu)$ and $L=C(c_\mu)$ are the constants
from the definition of the partitions of unity $\psi_{n,i}$.
We aim to show
that $g\in \mathbb{D}_X^\beta(\Lop_\beta f)$.
For this purpose we fix $k\in\Z$ and $x,y\in X$ such that $2^{-k-1}\le d(x,y)<2^{-k}$. 
Observe that $k\ge n_0$. 

Write $B_{n,i}=B(x_{n,i},2^{-n})$. From the definition of $\Lop_\beta$ we obtain
\begin{equation}\label{e.h1}
\lvert \Lop_\beta f(x)-\Lop_\beta f(y)\rvert \le  
\sum_{n=n_0}^\infty 2^{-\beta n}\sum_{i\in P_n}\lvert \psi_{n,i}(x)-\psi_{n,i}(y)\rvert\vint_{3B_{n,i}} f_n(z)\,d\mu(z)\,.
\end{equation}
We split the outer summation in two parts. First,
\begin{equation}\label{e.h2}
\begin{split}
&\sum_{n=k}^\infty 2^{-\beta n}\sum_{i\in P_n}\lvert \psi_{n,i}(x)-\psi_{n,i}(y)\rvert\vint_{3B_{n,i}} f_n(z)\,d\mu(z)\\
&\qquad \le 2^{-k\beta }\sum_{n=k}^\infty 2^{\beta (k-n)}\sum_{i\in P_n}(\psi_{n,i}(x)+\psi_{n,i}(y))\vint_{3B_{n,i}} f_n(z)\,d\mu(z)\\
&\qquad \le c_\mu 2^\beta d(x,y)^\beta\sum_{n=k}^\infty 2^{\beta (k-n)}
\sum_{i\in P_n}\bigl(\mathbf{1}_{6B_{n,i}}(x)+\mathbf{1}_{6B_{n,i}}(y)\bigr)\vint_{6B_{n,i}} f_n(z)\,d\mu(z)\\
&\qquad \le c_\mu 2^\beta d(x,y)^\beta\sum_{n=k}^\infty 2^{\beta (k-n)}
\sum_{i\in P_n}\bigl(\mathbf{1}_{6B_{n,i}}(x)M^*f_n(x)+\mathbf{1}_{6B_{n,i}}(y)M^*f_n(y)\bigr)\\
&\qquad \le Nc_\mu 2^\beta d(x,y)^\beta\sum_{n=k}^\infty 2^{\beta (k-n)}(M^*f_n(x)+M^*f_n(y))\,.
\end{split}
\end{equation}
Next we assume that $n_0\le n\le k-1$ and $i\in P_n$. 
If $\psi_{n,i}(x)\not=0$, then we have
\begin{align*}
&\lvert \psi_{n,i}(x)-\psi_{n,i}(y)\rvert\vint_{3B_{n,i}} f_n(z)\,d\mu(z)\\
&\qquad \le 2^n L d(x,y) \mathbf{1}_{6B_{n,i}}(x)c_\mu\vint_{6B_{n,i}} f_n(z)\,d\mu(z)\\
&\qquad \le 2^n L c_\mu d(x,y)^\beta 2^{-k(1-\beta)} \mathbf{1}_{6B_{n,i}}(x)M^*f_n(x)\,.
\end{align*}
On the other hand, if $\psi_{n,i}(y)\not=0$, then a similar computation shows that
\begin{align*}
&\lvert \psi_{n,i}(x)-\psi_{n,i}(y)\rvert\vint_{3B_{n,i}} f_n(z)\,d\mu(z)
\le 2^n Lc_\mu d(x,y)^\beta 2^{-k(1-\beta)} \mathbf{1}_{6B_{n,i}}(y)M^*f_n(y)\,.
\end{align*}
Thus in any case we have
\begin{equation}\label{e.h3}
\begin{split}
&\sum_{n=n_0}^{k-1} 2^{-\beta n}\sum_{i\in P_n}\lvert \psi_{n,i}(x)-\psi_{n,i}(y)\rvert\vint_{3B_{n,i}} f_n(z)\,d\mu(z)\\
&\quad \le 
 Lc_\mu d(x,y)^\beta \sum_{n=n_0}^{k-1} 2^{-\beta n+n-k(1-\beta)}\sum_{i\in P_n}
(\mathbf{1}_{6B_{n,i}}(x)M^*f_n(x)+\mathbf{1}_{6B_{n,i}}(y)M^*f_n(y))\\
&\quad \le 
 LNc_\mu d(x,y)^\beta \sum_{n=-\infty}^{k-1} 2^{(\beta-1)(k-n)}
(M^*f_n(x)+M^*f_n(y))\,.
\end{split}
\end{equation}
By combining inequalities \eqref{e.h1}, \eqref{e.h2} and \eqref{e.h3}, it follows that 
\[
\lvert \Lop_\beta f(x)-\Lop_\beta f(y)\rvert \le  d(x,y)^\beta(g_k(x)+g_k(y))\,,
\]
and therefore
$g\in\mathbb{D}_X^\beta(\Lop_\beta f)$.

Next, define $a=(a_k)_{k\in \Z}$, where
$a_k=Nc_\mu 2^\beta\cdot 2^{\beta k}$  if $k\le 0$ and 
$a_k=LNc_\mu 2^{(\beta-1)k}$ if $k> 0$.
Observe that $(g_k(x))_{k\in\Z}=a\star (M^*f_n(x))_{n\in\Z}$, where
$\star$ designates the discrete convolution
\[
\bigl(a\star (M^*f_n(x))_{n\in\Z}\bigr)_k
=\sum_{n\in\Z} a_{k-n}M^* f_n(x)\,,\qquad k\in\Z\,.
\]
Young's convolution inequality \cite[Corollary 20.14]{MR551496} implies for every $x\in X$ that
\begin{align*}
\big\lVert (g_k(x))_{k\in\Z}\big\rVert_{l^q(\Z)}
&=\big\lVert a\star (M^*f_n(x))_{n\in\Z}\big\rVert_{l^q(\Z)}\\
&\le \lVert a\rVert_{l^1(\Z)}\big\lVert (M^*f_n(x))_{n\in\Z}\big\rVert_{l^q(\Z)}\\
&\le C(L,N,c_\mu,\beta)\big\lVert (M^*f_n(x))_{n\in\Z}\big\rVert_{l^q(\Z)}\,.
\end{align*}
Hence, we see that
\[
\lvert \Lop_\beta f\rvert_{\dot{M}^{\beta}_{p,q}(X)}
\le \lVert (g_k)_{k\in\Z}\rVert_{L^p(X;l^q(\Z))}
\le C(L,N,c_\mu,\beta)\big\lVert (M^*f_n)_{n\in\Z}\big\rVert_{L^p(X;l^q(\Z))}\,.
\]
By the Fefferman--Stein vector-valued maximal function inequality \eqref{e.fs}, we obtain
\begin{align*}
\big\lVert (M^*f_n)_{n\in\Z}\big\rVert_{L^p(X;l^q(\Z))}
&=\Biggl(\int_X \Biggl(\sum_{n\in\Z} \left(M^*f_n(x)\right)^q\Biggr)^{p/q}\,d\mu(x)\Biggl)^{1/p}\\
&\le C(c_\mu,p,q)\Biggl(\int_X \Biggl(\sum_{n\in\Z} f_n(x)^q\Biggl)^{p/q}\,d\mu(x)\Biggl)^{1/p}\\
&=C(c_\mu,p,q)\lVert (f_n)_{n\in\Z}\rVert_{L^p(X;l^q(\Z))}
\end{align*}
(with the usual modification if $q=\infty$).
This concludes the proof, since the constants $L$ and $N$ only depend on $c_\mu$. 
\end{proof}

The following lemma gives for the discrete potential operator $\Lop_\beta$ a local embedding, 
which can be viewed as a Sobolev type inequality.

\begin{lemma}\label{l.pot_poincare}
Let $1<p<\infty$, $1<q\le \infty$, and $\beta>0$.
Assume that $\mu$ satisfies the quantitative reverse doubling condition~\eqref{e.reverse_doubling}
for some exponent  $\sigma>\beta p$, and
let $f=(f_n)_{n=n_0}^\infty$ be a sequence of nonnegative Borel functions in $X$. 
Then there exists
a constant  $C=C(c_\sigma,p,q,\sigma,\beta,c_\mu)$  such that
\[
\left(\int_{B(x_0,r)} (\Lop_\beta f(x))^p\,d\mu(x)\right)^{1/p} \le Cr^{\beta}\lVert f\rVert_{L^p(X;l^q)}
\]
for every $x_0\in X$ and $r>0$.
\end{lemma}

\begin{proof}
Fix $x_0\in X$ and $r>0$.
If $X$ is bounded, then we may clearly assume 
that $0<r\le 2^{-n_0}$. 
Let $k\in\mathbb{Z}$ be such that $2^{-k-1}< r\le 2^{-k}$
and let $\psi_{n,i}$ be as in the definition of $\Lop_\beta$. 
We write $B_{n,i}=B(x_{n,i},2^{-n})$
and $\Lop_\beta f=F+G$, where
\[
F(x)=\sum_{n=k}^\infty 2^{-\beta n} \sum_{i\in P_n}\psi_{n,i}(x)\vint_{3B_{n,i}} f_n(y)\,d\mu(y)
\]
and
\[
G(x)=\sum_{n=n_0}^{k-1} 2^{-\beta n} \sum_{i\in P_n}\psi_{n,i}(x)\vint_{3B_{n,i}} f_n(y)\,d\mu(y)
\]
for every $x\in X$.

Fix a point $x\in B(x_0,r)$.  We estimate $F(x)$ by using H\"older's inequality, with exponents $q$ and $q'=\frac{q}{q-1}$, 
\begin{align*}
F(x)&=2^{-\beta k}\sum_{n=k}^\infty 2^{-\beta(n-k)} \sum_{i\in P_n}\psi_{n,i}(x)\vint_{3B_{n,i}} f_n(y)\,d\mu(y)\\
&\le 2^{-\beta k}\left(\sum_{n=k}^\infty 2^{-\beta(n-k)q'}\right)^{1/q'}
  \left(\sum_{n=k}^\infty\left(\sum_{i\in P_n}\psi_{n,i}(x)\vint_{3B_{n,i}} f_n(y)\,d\mu(y)\right)^q\right)^{1/q}\\
&\le C(q,\beta,c_\mu) 2^{-\beta k} \left(\sum_{n=k}^\infty
\left(\sum_{i\in P_n}\mathbf{1}_{6B_{n,i}}(x) \vint_{6B_{n,i}} f_n(y)\,d\mu(y)\right)^q\right)^{1/q}\\
&\le 
C(q,\beta,c_\mu,N) 2^{-\beta k} \left(\sum_{n=k}^\infty\left(M^*f_n(x)\right)^q\right)^{1/q}\,,
\end{align*}
where $N=C(c_\mu)$ is the constant from~\eqref{e.overlap}.
By the Fefferman--Stein vector-valued maximal function inequality \eqref{e.fs}, we obtain
\begin{equation}\label{e.f1}
\begin{split}
\int_{B(x_0,r)} (F(x))^p\,d\mu(x)&\le C(p,q,\beta,c_\mu,N) 2^{-\beta pk}  \int_{X} \left(\sum_{n=k}^\infty\left(M^*f_n(x)\right)^q\right)^{p/q}\,d\mu(x)\\
&\le C(p,q,\beta,c_\mu,N) 2^{-\beta pk} \int_{X} \left(\sum_{n=k}^\infty f_n(x)^q\right)^{p/q}\,d\mu(x)
\end{split}
\end{equation}
(with an obvious modification if $q=\infty$).

On the other hand, from the quantitative reverse doubling condition~\eqref{e.reverse_doubling} with exponent $\sigma>\beta p$, it follows that
\[
1\le c_\sigma\frac{\mu(B(x_0,2^{-n}))}{\mu(B(x_0,2^{-k}))}\left(\frac{2^{-k}}{2^{-n}}\right)^\sigma
=c_\sigma2^{\sigma(n-k)}\frac{\mu(B(x_0,2^{-n}))}{\mu(B(x_0,2^{-k}))}
\]
whenever  $n_0\le n\le k-1$. 
Fix $x\in B(x_0,r)$. We estimate $G(x)$ as follows
\begin{align*}
G(x)&=2^{-k\beta}\sum_{n=n_0}^{k-1} 2^{-\beta (n-k)} \sum_{i\in P_n}\psi_{n,i}(x)\vint_{B(x_{n,i},3\cdot 2^{-n})} f_n(y)\,d\mu(y) \\
&\le c_\sigma^{1/p} 2^{-k\beta}\sum_{n=n_0}^{k-1} 2^{(\sigma/p-\beta) (n-k)}\left(\frac{\mu(B(x_0,2^{-n}))}{\mu(B(x_0,2^{-k}))}\right)^{1/p} 
\sum_{i\in P_n}\psi_{n,i}(x)\vint_{3B_{n,i}} f_n(y)\,d\mu(y)\,.
\end{align*}
Writing $2^{(\sigma/p-\beta) (n-k)}=2^{(\sigma/p-\beta) (n-k)/p'}2^{(\sigma/p-\beta) (n-k)/p}$
and using 
H\"older's inequality, we have
\begin{equation}\label{e.g1}
\begin{split}
G(x)&\le C(c_\sigma,p)2^{-k\beta}\left(\sum_{n=n_0}^{k-1} 2^{(\sigma/p-\beta) (n-k)}\right)^{1/p'}\\
&\qquad  \times \left(\sum_{n=n_0}^{k-1} 2^{(\sigma/p-\beta) (n-k)}\frac{\mu(B(x_0,2^{-n}))}{\mu(B(x_0,2^{-k}))} \left(\sum_{i\in P_n}\psi_{n,i}(x)\vint_{3B_{n,i}} f_n(y)\,d\mu(y)\right)^p\right)^{1/p}\,.
\end{split}
\end{equation}
Since $n_0\le n\le k-1$ and $x\in B(x_0,r)$ with $r\le 2^{-k}$, we obtain
\begin{equation}\label{e.g2}
\begin{split}
&\left(\sum_{i\in P_n}\psi_{n,i}(x)\vint_{3B_{n,i}} f_n(y)\,d\mu(y)\right)^p\\
&\qquad\qquad \le \left(\sum_{i\in P_n}\mathbf{1}_{6B_{n,i}}(x)\vint_{3B_{n,i}} f_n(y)\,d\mu(y)\right)^p\\
& \qquad\qquad \le C(c_\mu,p)\left(\sum_{i\in P_n}\mathbf{1}_{6B_{n,i}}(x)\vint_{B(x_0,10\cdot 2^{-n})} f_n(y)\,d\mu(y)\right)^p\\
&\qquad\qquad \le C(c_\mu,p,N)\left(\vint_{B(x_0,10\cdot 2^{-n})} f_n(y)\,d\mu(y)\right)^p\\
&\qquad\qquad \le C(c_\mu,p,N)\vint_{B(x_0,10\cdot 2^{-n})} f_n(y)^p\,d\mu(y)\,.
\end{split}
\end{equation}
By combining the  estimates \eqref{e.g1} and \eqref{e.g2}, we see that
\begin{align*}
G(x)\le C(c_\sigma,p,\sigma,\beta,c_\mu,N)2^{-k\beta}
\left(\sum_{n=n_0}^{k-1}\frac{2^{(\sigma/p-\beta) (n-k)}}{\mu(B(x_0,2^{-k}))} \int_{B(x_0,10\cdot 2^{-n})} f_n(y)^p\,d\mu(y)\right)^{1/p}\,.
\end{align*}
It follows that
\begin{equation}\label{e.g3}
\begin{split}
\int_{B(x_0,r)} & G(x)^p\,d\mu(x)\\
&\le C(c_\sigma,p,\sigma,\beta,c_\mu,N)2^{-kp\beta}
\sum_{n=n_0}^{k-1}2^{(\sigma/p-\beta) (n-k)}\int_{B(x_0,10\cdot 2^{-n})} f_n(y)^p\,d\mu(y)\\
&\le C(c_\sigma,p,\sigma,\beta,c_\mu,N)2^{-kp\beta}
\sum_{n=n_0}^{k-1}2^{(\sigma/p-\beta) (n-k)}\int_{X} \left(\sum_{m=n_0}^{k-1} f_m(y)^q\right)^{p/q}\,d\mu(y)\\
&\le C(c_\sigma,p,\sigma,\beta,c_\mu,N)2^{-kp\beta}
\int_{X} \left(\sum_{n=n_0}^{k-1} f_n(y)^q\right)^{p/q}\,d\mu(y)
\end{split}
\end{equation}
(with an obvious modification if $q=\infty$).

By combining the estimates \eqref{e.f1} and \eqref{e.g3}, we get
\begin{align*}
\int_{B(x_0,r)} (\Lop_\beta f(x))^p\,d\mu(x)&\le 2^{p-1}\int_{B(x_0,r)}  (F(x))^p\,d\mu(x)+2^{p-1}\int_{B(x_0,r)}  (G(x))^p\,d\mu(x)\\
&\le C(c_\sigma,p,q,\sigma,\beta,c_\mu,N)2^{-kp\beta}
\int_{X} \left(\sum_{n=n_0}^{\infty} f_n(y)^q\right)^{p/q}\,d\mu(y)\\
&\le C(c_\sigma,p,q,\sigma,\beta,c_\mu,N)r^{\beta p} \lVert f\rVert_{L^p(X;l^q)}^p\,.
\end{align*}
This concludes the proof since $N=C(c_\mu)$.  
\end{proof}

\section{Capacity with respect to a potential}\label{s.pot_cap}

Next we define the potential $\Hop_\beta$ and the 
associated Triebel--Lizorkin capacity $\Cp^\beta_{p,q}$.
The main goal in this section is to show that the
capacities $\Cp^\beta_{p,q}$ and $\capm{\beta}{p}{q}$ are equivalent,
as was stated in Theorem~\ref{t.equiv_intro} in the introduction.
The proof of Theorem~\ref{t.equiv_intro} follows by combining
Theorems~\ref{t.equiv_bp} and~\ref{t.comparison_dc} that are
presented at the end of this section.

The discrete potential operator $\Lop_\beta$ serves as an important technical tool in the proof of
Theorem~\ref{t.equiv_bp}. More precisely, in Lemmas~\ref{l.disc_est} and~\ref{l.hbest} we connect
$\Lop_\beta$ with the potential $\Hop_\beta$ and the capacity $\capm{\beta}{p}{q}$, respectively,
and these two elements are then brought together in the proof of Theorem~\ref{t.equiv_bp}. 
On the other hand, Theorem~\ref{t.comparison_dc} follows by a much more 
direct argument.

\begin{definition}\label{d.hbeta}
Let $\beta >0$ and let $f=(f_n)_{n=n_0}^\infty$ be a sequence of nonnegative Borel functions in $X$. The potential $\Hop_\beta f$ is defined by
\[
\Hop_\beta f(x)= \sum_{n=n_0}^\infty  2^{-\beta n}\int_{X} \frac{\mathbf{1}_{B(y,2^{-n})}(x)}{\mu(B(y,2^{-n}))} f_n(y)\,d\mu(y)\,,\qquad x\in X\,.
\]
\end{definition}

We use potentials $\Hop_\beta$ to define the capacity $\Cp^\beta_{p,q}$,
by adapting the treatment in \cite[Section 4.4]{MR1411441} to the metric space setting.
This approach is a Triebel--Lizorkin variant of the so-called Meyer's theory of $L^p$-capacities, 
that was initially developed in~\cite{MR277741}. 
See also~\cite[Section 3.2]{MR1774162} for a clear presentation of the $L^p$-theory
in (weighted) Euclidean spaces.

\begin{definition}\label{d.cpq}
Let $1<p<\infty$, $1\le q\le \infty$, and $\beta>0$. The $(\beta,p,q)$-capacity of a set $E\subset X$ is defined by
\[
\Cp^\beta_{p,q}(E) = \inf\bigl\{\lVert f\rVert_{L^p(X;l^q)}^p\,:\, 
f=(f_n)_{n=n_0}^\infty\ge 0\text{ and } \Hop_\beta f(x)\ge 1\text{ for all }x\in E \bigr\}\,,
\]
where $(f_n)_{n=n_0}^\infty\ge 0$ in the infimum means that
$(f_n)_{n=n_0}^\infty$  
is a sequence of nonnegative Borel functions in $X$.
\end{definition}

Observe that if $F\subset E\subset X$ and $1\le \tau\le q\le \infty$,
then 
\begin{equation}\label{e.Cap_E_comp}
\Cp^\beta_{p,q} (F)\le \Cp^\beta_{p,q} (E)\le   \Cp^{\beta}_{p,\tau} (E).
\end{equation}

 We have the following  pointwise comparison between
the potentials $\Hop_{\beta} f$ and $\Lop_\beta f$.

\begin{lemma}\label{l.disc_est}
Let $\beta>0$ and let $f=(f_n)_{n=n_0}^\infty$ be a sequence of nonnegative Borel functions in~$X$.  Then there
exists a constant $C=C(c_\mu)>0$ such that
\[
\Hop_{\beta} f(x)\le C\Lop_\beta f(x)
\]
for every $x\in X$. 
\end{lemma}

\begin{proof}
Let $x\in X$, let $n\in\Z$ be such that $n\ge n_0$, and let
functions $\psi_{n,i}$, $i\in P_n$, be as in the definition
of $\Lop_\beta f$, see Definition~\ref{d.dpotential}.
By the covering property~\eqref{e.coverX},  
there exists $j\in P_n$ such that $x\in B(x_{n,j},2^{-n})$.
Then $B(x,2^{-n})\subset B(x_{n,j},3\cdot 2^{-n})\subset B(y,5\cdot 2^{-n})$ for every $y\in B(x,2^{-n})$.
Observe also that $\mathbf{1}_{B(y,2^{-n})}(x)=\mathbf{1}_{B(x,2^{-n})}(y)$ for every $y\in X$.
It follows that
\begin{align*}
2^{-\beta n}\int_{X} \frac{\mathbf{1}_{B(y,2^{-n})}(x)}{\mu(B(y,2^{-n}))} f_n(y)\,d\mu(y)
&\le  c_\mu^3 2^{-\beta n}\vint_{B(x_{n,j},3\cdot 2^{-n})} f_n(y)\,d\mu(y)\\
&\le  \kappa^{-1}c_\mu^3 2^{-\beta n} \psi_{n,j}(x)\vint_{B(x_{n,j},3\cdot 2^{-n})} f_n(y)\,d\mu(y)\\
&\le \kappa^{-1}c_\mu^3 2^{-\beta n}\sum_{i\in P_n} \psi_{n,i}(x)\vint_{B(x_{n,i},3\cdot 2^{-n})} f_n(y)\,d\mu(y)\,,
\end{align*}
where $\kappa=C(c_\mu)$ is as in the definition of the functions $\psi_{n,i}$.
By summing over all $n\ge n_0$, we find that
\begin{align*}
\mathcal{H}_\beta f(x)&=\sum_{n=n_0}^\infty2^{-\beta n}\int_{X} \frac{\mathbf{1}_{B(y,2^{-n})}(x)}{\mu(B(y,2^{-n}))}f_n(y)\,d\mu(y)
\\&\le \kappa^{-1}c_\mu^3 \sum_{n=n_0}^\infty2^{-\beta n}\sum_{i\in P_n} \psi_{n,i}(x)\vint_{B(x_{n,i},3\cdot 2^{-n})} f_n(y)\,d\mu(y)=\kappa^{-1}c_\mu^3\Lop_\beta f(x)\,,
\end{align*}
and this concludes the proof. 
\end{proof}

Next we connect the relative Haj{\l}asz--Triebel--Lizorkin capacity to
the $\Lop_\beta$ potentials.

\begin{lemma}\label{l.hbest}
Let $1<p<\infty$, $1<q\le \infty$, $0<\beta <1$, and $\Lambda\ge  2$.
Assume that $\mu$ satisfies the quantitative reverse doubling condition~\eqref{e.reverse_doubling} for some exponent  $\sigma>\beta p$. Let $B\subset X$ be a ball, let $E\subset \iol{B}$ be a compact set, and let
$f=(f_n)_{n=n_0}^\infty$ is  sequence of nonnegative Borel functions in $X$ such
that $\Lop_\beta f(x)\ge 1$ for every $x\in E$. Then
there exists a constant  $C=C(c_\sigma,p,q,\sigma,\beta,c_\mu,\Lambda)>0$  such that
\begin{equation}\label{e.fo}
\capm{\beta}{p}{q}(E, 2B,\Lambda B)\le C\lVert f\rVert_{L^p(X;l^q)}^p\,.
\end{equation}
\end{lemma}

\begin{proof}
Clearly, we may assume that 
the right-hand side of  \eqref{e.fo} is finite.
First we consider the case where $\Lop_\beta f$ is continuous.
We truncate $\Lop_\beta f$ and then use the fractional type Leibniz rule, as follows.
Write $B=B(x_0,r)\subset X$, with $x_0\in X$ and $r>0$, and 
let 
\[
\eta(x)=\max\biggl\{0,1-\frac{\dist(x,B)}{r}\biggr\}
\] 
for every $x\in X$. 
Then 
\[
\lvert\eta(x)-\eta(y)\rvert\leq \frac{d(x,y)}{r},\qquad \text{for every }x,y\in X\,, 
\] 
$0\leq \eta \leq 1$ in $X$, $\eta=1$ in the closure $\iol{B}$, and $\eta=0$ in $X\setminus 2B$. The function 
$\varphi=\eta \Lop_\beta f$ is continuous in $X$. Since $E\subset \iol{B}$ and $\Lop_\beta f\ge 1$ in $E$, we have
$\varphi\ge 1$ in $E$, and clearly $\varphi=0$ in $X\setminus 2B$.
Let $g=(g_k)_{k\in\N}\in\mathbb{D}^\beta_X( \Lop_\beta f)$.
Lemma~\ref{LemWithLipForTL} implies that $(\rho_k)_{k\in\Z}\in\mathbb{D}_X^\beta(\varphi)$, where 
\[
\rho_k=\big(g_k\Vert \eta\Vert_{\infty}+2^{k(\beta-1)}r^{-1} \Lop_\beta f \big)\mathbf{1}_{\{x\in X\,:\,\eta(x)\not=0\}}\,,\quad \text{ for every } k\in\Z\,.
\]
Let $m$ be the smallest integer
such that $2^{-m-1}\le 2\Lambda r$.
Observe that $(\mathbf{1}_{k\ge m}\rho_k)_{k\in\Z}\in  \mathbb{D}_{\Lambda B}^\beta(\varphi)$.

By Definition~\ref{d.capacity} of the capacity $\capm{\beta}{p}{q}$, 
we have 
\[\begin{split}
  &\capm{\beta}{p}{q}(E, 2B,\Lambda B) \leq 
  \lvert \varphi\rvert_{\dot{M}^{\beta}_{p,q}(\Lambda B)}^p
   \le \int_{{\Lambda B}} \left(\sum_{k\ge m} \rho_k(x))^q\right)^{p/q}\,d\mu(x)\\
  &\quad \le C(p) \int_{2B} \left(\sum_{k\ge m} g_k(x))^q\right)^{p/q}\,d\mu(x)+C(p)\int_{2B} \left(\sum_{k\ge m} \left(2^{k(\beta-1)}r^{-1} \Lop_\beta f(x)\right)^q\right)^{p/q}\,d\mu(x)\\
  &\quad \le C(p)\int_{X} \left(\sum_{k\in \Z} g_k(x))^q\right)^{p/q}\,d\mu(x)
 + C(p)r^{-p}\left(\sum_{k\ge m} 2^{kq(\beta-1)}\right)^{p/q}\int_{2B} \left(\Lop_\beta f(x)\right)^p\,d\mu(x)\\
 &\quad \le C(p)\int_{X} \left(\sum_{k\in \Z} g_k(x))^q\right)^{p/q}\,d\mu(x)
 + C(p,q,\beta)r^{-p}2^{mp(\beta-1)}\int_{2B} \left(\Lop_\beta f(x)\right)^p\,d\mu(x)\\
  &\quad \le C(p)\int_{X} \left(\sum_{k\in \Z} g_k(x))^q\right)^{p/q}\,d\mu(x)
 + C(p,q,\beta,\Lambda)r^{-\beta p}\int_{2B} \left(\Lop_\beta f(x)\right)^p\,d\mu(x)\,.
\end{split}\]
By taking infimum over all $g$ as above, we obtain
\[
\capm{\beta}{p}{q}(E, 2B,\Lambda B) 
\le C(p)\lvert \Lop_\beta f\rvert_{\dot{M}^{\beta}_{p,q}(X)}^p
 + C(p,q,\beta,\Lambda)r^{-\beta p}\int_{2B} \left(\Lop_\beta f(x)\right)^p\,d\mu(x)\,,
\]
and the claim in the present case follows from Lemmas~\ref{l.pot_norm_est} and~\ref{l.pot_poincare}.

Next we consider the case where the function $\Lop_\beta f$ is not continuous. 
For every $k\in\N$, we let $h_k=(h_{k,n})_{n=n_0}^\infty$, where 
$h_{k,n}=\mathbf{1}_{\lvert n\rvert\le k}f_n$ for every $n\ge n_0$. 
Then $(\Lop_\beta h_k)_{k\in\N}$ is an increasing sequence
of continuous real-valued functions in $X$ that 
converges to $\Lop_\beta f$ pointwise in $X$.
We consider the open sets \[
E_k=\left\{x\in X\,:\, \Lop_\beta h_k(x)>\frac{1}{2}\right\}\,,\qquad k\in\N\,.
\]
Since $\Lop_\beta f(x)\ge 1$ for every  $x\in E$, we see that
$E\subset \bigcup_{k\in\N} E_k$. By the compactness of $E$, there
exists a finite subcover, and since $E_n\subset E_{n+1}$ for every $n\in\N$, 
there exists  $k\in\N$ such that $E\subset E_k$.
Thus we have $\Lop_\beta (2h_k)> 1$ in $E$, and 
by the first part of the proof
\[
\capm{\beta}{p}{q}(E, 2B,\Lambda B)\le C\lVert 2h_k\rVert_{L^p(X;l^q)}^p
\le 2^pC\lVert f\rVert_{L^p(X;l^q)}^p\,.
\]
This shows that inequality \eqref{e.fo} holds also in the non-continuous case,
and the proof is complete.
\end{proof}

We now have all the tools needed for the comparison of the two capacities. We begin
with the second inequality in~\eqref{e.equiv_intro}.

\begin{theorem}\label{t.equiv_bp}
Let $1<p<\infty$, $1<q\le\infty$, $0<\beta <1$, and $\Lambda\ge 2$, and
assume that $\mu$ satisfies the quantitative reverse doubling condition~\eqref{e.reverse_doubling} 
for some exponent  $\sigma>\beta p$.
Let $B\subset X$ be a ball and $E\subset \iol{B}$ be a compact set.
Then there exists a constant  $C=C(c_\sigma,p,q,\sigma,\beta,c_\mu,\Lambda)>0$ such that
\begin{equation}\label{e.comp_caps}
\capm{\beta}{p}{q}(E, 2B,\Lambda B)\le  C \Cp^\beta_{p,q}(E)\,.
\end{equation}
\end{theorem}

\begin{proof}
We may assume that $\Cp^\beta_{p,q}(E)<\infty$. Let
$f=(f_n)_{n=n_0}^\infty$ be a sequence of nonnegative Borel functions in $X$ such that $\Hop_\beta f\ge 1$ in $E$.
By Lemma~\ref{l.disc_est}, there is a constant $C=C(c_\mu)>0$ such that
$C\Lop_\beta f(x)\ge 1$ for every $x\in E$.
Hence $\Lop_\beta (Cf)(x)\ge 1$ for every $x\in E$ and therefore Lemma \ref{l.hbest} implies that
\[
\capm{\beta}{p}{q}(E, 2B,\Lambda B)\le C(c_\sigma,p,q,\sigma,\beta,c_\mu,\Lambda)\lVert Cf\rVert_{L^p(X;l^q)}^p\le C(c_\sigma,p,q,\sigma,\beta,c_\mu,\Lambda)\lVert f\rVert_{L^p(X;l^q)}^p\,.
\]
Inequality \eqref{e.comp_caps} follows by taking infimum over all sequences $f$ as above.
\end{proof}

The converse inequality, that is, the first inequality in~\eqref{e.equiv_intro}, 
requires further restrictions on $\Lambda$ and $r$, but otherwise the assumptions are less restrictive
than in Theorem~\ref{t.equiv_bp}.

\begin{theorem}\label{t.comparison_dc}
Let $1<p<\infty$, $1\le q\le\infty$, $0<\beta <1$, and $\Lambda\ge 41$, and
assume that $\mu$ satisfies the
reverse doubling condition \eqref{e.rev_dbl_decay}, with a constant $0<c_R<1$. 
Assume that $B=B(x_0,r)$, with $x_0\in X$ and $0<r<(1/80)\diam(X)$,
and let $E\subset \iol{B}$ be a closed set.
Then there exists a constant  $C=C(p,c_\mu,c_R)>0$ such that
\[
\Cp^\beta_{p,q}(E)\le C\capm{\beta}{p}{q}(E, 2B,\Lambda B)\,.
\]
\end{theorem}

\begin{proof}
The proof is similar to that of Theorem \ref{t.HC_bp}.
Let $\varphi\colon X\to \R$ be a continuous function  such that $\varphi(x)\ge 1$ for every $x\in E$ and 
$\varphi(x)=0$ for every $x\in X\setminus 2B$, and let $g=(g_k)_{k\in\Z}\in\mathbb{D}^\beta_{\Lambda B}(\varphi)$.
By replacing $\varphi$ with $\max\{0,\min\{\varphi,1\}\}$, if necessary, we may assume that $0\le \varphi \le 1$.

Let $x\in E$ and let $m\in \Z$ be such that $2^{-m-1}\le 5r<2^{-m}$.
Arguing as in the proof of Theorem \ref{t.HC_bp}
and using the inclusion $B(x,2^{-n})\subset B(x_0,\Lambda r)$ if $n\ge m-2$,  we obtain
\begin{align*}
0<1-c_R&\le C(c_\mu,c_R)\sum_{n=m-2}^\infty 2^{-n \beta}\;\sum_{k= n-1}^{n+2}\vint_{B(x,2^{-n})}  \mathbf{1}_{\Lambda B}(y)g_k(y)\,d\mu(y)\,.
\end{align*}
Define $f=(f_n)_{n\ge n_0}$, where $f_n=\mathbf{1}_{\Lambda B}\sum_{k= n-1}^{n+2}g_k$ for each $n\ge n_0$.  
Since $m-2>n_0$, we have
\begin{align*}
0<1-c_R&\le C(c_\mu,c_R)\sum_{n=n_0}^\infty 2^{-n \beta}\;\vint_{B(x,2^{-n})}   f_n(y)\,d\mu(y)\\
&\le C(c_\mu,c_R)\sum_{n=n_0}^\infty 2^{-n \beta}\;\int_{X} \frac{\mathbf{1}_{B(y,2^{-n})}(x)}{\mu(B(y,2^{-n}))}  f_n(y)\,d\mu(y)\\
&=C(c_\mu,c_R)\Hop_\beta f(x)\,.
\end{align*}
Hence, there exists a constant $C=C(c_\mu,c_R)>0$ such that  $\Hop_\beta(Cf)(x)\ge 1$ for every $x\in E$.
Thus
\[
\Cp^\beta_{p,q}(E) \le  \lVert Cf\rVert_{L^p(X;l^q)}^p=C^p\lVert f\rVert_{L^p(X;l^q)}^p
\le (4C)^p\lVert g\rVert_{L^p(\Lambda B;l^q)}^p\,,
\]
and the claim follows by taking infimum first over all $g$ and the over all $\varphi$ as above.
\end{proof}

\section{Dual characterization of capacity}\label{s.Riesz_capacity}

We now turn to the dual characterization of the capacity $\Cp^\beta_{p,q}$ given in terms of measures,
see Theorem~\ref{t.n_cap_dual} below.
In Euclidean spaces, this type of results have a long history, see for instance \cite{MR1097178, MR1411441, MR277741}
and the references therein. We prove the dual characterization 
by adapting the proof of equality (4.4.7) in \cite{MR1411441} to the setting of metric measure spaces.
A fundamental tool in the proof is the Minimax theorem~\cite[Theorem~2.4.1]{MR1411441}. 
The somewhat technical task of verifying the assumptions of the Minimax theorem 
in our setting is carried out separately in Lemma~\ref{l.m_assumptions}.

\begin{definition}
Let $\beta>0$ and let
$\nu$ be a nonnegative Borel measure in $X$. We define 
\begin{equation}\label{e.pot}
\widecheck{\mathcal{H}}_\beta\nu(x)=\left(  2^{-\beta n} \frac{\nu(B(x,2^{-n}))}{\mu(B(x,2^{-n}))} \right)_{n=n_0}^\infty\,,\qquad x\in X.
\end{equation}
\end{definition}

Of particular interest to us are the nonnegative Radon measures
on a compact set $E\subset X$.

\begin{definition}
Let $E\subset X$ be a compact set.
We say that $\nu$ is a nonnegative Radon measure on $E$, if
$\nu$ is a nonnegative Borel measure on $E$ that satisfies the following properties: 
\begin{itemize}
\item $\nu(E)<\infty$.
\item For every Borel set $F$ in $E$, we have 
\[\nu(F)=\inf\{\nu(U)\,:\,F\subset U\text{ and } U \text{ is open in } E\}\,.\]
\item For every Borel set $F$ in $E$, we have
\[
\nu(F)=\sup\{\nu(K)\,:\,K\subset F\text{ and } K \text{ is compact in }E\}\,.
\]
\end{itemize} 
We denote by $\mathcal{M}^+(E)$ the set of all nonnegative Radon measures on $E$.
If $f$ is a Borel function on $X$ and $\nu\in\mathcal{M}^+(E)$, we 
define
\[
\int_X f\,d\nu = \int_E f|_E\,d\nu\,.
\]
We also define $\nu(A)=\nu(A\cap E)$ if $A\subset X$ is a Borel set, and this
extends $\nu$ as a nonnegative Borel measure in $X$.
\end{definition}

We are ready to state and prove the dual characterization of the capacity.

\begin{theorem}\label{t.n_cap_dual}
Let $1<p<\infty$, $1\le q\le \infty$, and $\beta>0$,
and let $E\subset X$ be a compact set. 
Then
\[
\Cp^\beta_{p,q}(E)^{1/p}  =  \sup
\Biggl\{\frac{\nu(E)}{\big\lVert \widecheck{\mathcal{H}}_\beta \nu\big\rVert_{L^{p'}(X;l^{q'})}}\,:\, \nu\in\mathcal{M}^+(E) \Biggr\}\,.
\]
\end{theorem}

\begin{proof}
Let $\nu\in\mathcal{M}^+(E)$, let $f=(f_n)_{n=n_0}^\infty$ be a sequence of nonnegative Borel functions,
and write
\[
\mathcal{E}_\beta(\nu,f)=\int_{X} \Hop_\beta f(x)\,d\nu(x)\,.
\]
Using Definition~\ref{d.hbeta} and Fubini's theorem, we have
\begin{equation}\label{e.duality}
\begin{split}
\mathcal{E}_\beta(\nu,f)
&=\sum_{n=n_0}^\infty  2^{-\beta n}\int_{X} \int_{X} \frac{\mathbf{1}_{B(y,2^{-n})}(x)}{\mu(B(y,2^{-n}))} f_n(y)\,d\mu(y)\,d\nu(x)\\
&=\sum_{n=n_0}^\infty  2^{-\beta n}\int_{X} \frac{f_n(y)}{\mu(B(y,2^{-n}))} \int_{X} \mathbf{1}_{B(y,2^{-n})}(x)\,d\nu(x)\, d\mu(y)\\
&=\int_{X} \sum_{n=n_0}^\infty  2^{-\beta n}\frac{\nu(B(y,2^{-n}))}{\mu(B(y,2^{-n}))}\, f_n(y)\,d\mu(y)\\
&=\int_{X} \langle \widecheck{\mathcal{H}}_\beta\nu(y),f(y)\rangle \,d\mu(y)\,.
\end{split}
\end{equation}
Now, assume that
$\Hop_\beta f(x)\ge 1$ for all $x\in E$.
Then, by using also
H\"older's inequality twice, we get
\[
\nu(E)\le \mathcal{E}_\beta(\nu,f)\le \big\lVert \widecheck{\mathcal{H}}_\beta \nu\big\rVert_{L^{p'}(X;l^{q'})}\lVert f\rVert_{L^{p}(X;l^{q})}\,.
\]
By taking infimum over all sequences $f$ with $\Hop_\beta f\ge 1$ in $E$, we thus obtain
\[
\frac{\nu(E)}{\big\lVert \widecheck{\mathcal{H}}_\beta \nu\big\rVert_{L^{p'}(X;l^{q'})}}\le \Cp^\beta_{p,q}(E)^{1/p}\,.
\]
Taking supremum over all $\nu\in\mathcal{M}^+(E)$ gives
\[
\sup
\Biggl\{\frac{\nu(E)}{\big\lVert \widecheck{\mathcal{H}}_\beta \nu\big\rVert_{L^{p'}(X;l^{q'})}}\,:\, \nu\in\mathcal{M}^+(E) \Biggr\}\le \Cp^\beta_{p,q}(E)^{1/p}\,.
\]

For the  converse inequality, we  will use the Minimax theorem \cite[Theorem 2.4.1]{MR1411441}; we refer to Lemma~\ref{l.m_assumptions} for further details. Define
\[Z=\{\nu\in\mathcal{M}^+(E)\,:\,\nu(E)=1\}\]
and
\[
W=\big\{f\in L^p(X;l^q)\,:\,f=(f_n)_{n=n_0}^\infty\ge 0 \text{ and }\lVert f\rVert_{L^p(X;l^q)}\le 1\big\}\,.
\]
Observe from \cite[Lemma 2.3]{MR2542655} and \eqref{e.duality} that for $\nu\in Z$ we have
\[
\frac{\big\lVert \widecheck{\mathcal{H}}_\beta \nu\big\rVert_{L^{p'}(X;l^{q'})}}{\nu(E)} = \sup_{f\in W} \int_{X}\langle \widecheck{\mathcal{H}}_\beta\nu(y),f(y)\rangle \,d\mu(y)=
 \sup_{f\in W}\mathcal{E}_\beta(\nu,f)\,.
\]
By taking infimum over all $\nu\in Z$, we get
\begin{equation}\label{e.connect}
\inf_{\nu\in Z} \frac{\big\lVert \widecheck{\mathcal{H}}_\beta \nu\big\rVert_{L^{p'}(X;l^{q'})}}{\nu(E)}=\inf_{\nu\in Z}\sup_{f\in W}\mathcal{E}_\beta(\nu,f)\,.
\end{equation}
On the other hand, since the Dirac delta measure $\delta_x$ belongs to $Z$ for any $x\in E$, we have 
\[
\inf_{\nu\in Z}\mathcal{E}_\beta(\nu,f)=\inf_{\nu\in Z} \int_{X} \Hop_\beta f(x)\,d\nu(x)\leq\inf_{x\in E}\Hop_\beta f(x)
\le \frac{\inf_{x\in E}\Hop_\beta f(x)}{\lVert f\rVert_{L^p(X;l^q)}}
\]
for every $f\in W$.
By scaling and Definition~\ref{d.cpq}, it follows that
\begin{equation}\label{e.connect2}
\begin{split}
	\sup_{f\in W}\inf_{\nu\in Z}\mathcal{E}_\beta(\nu,f)
&\le \sup_{f\in W}\frac{\inf_{x\in E}\Hop_\beta f(x)}{\lVert f\rVert_{L^p(X;l^q)}}=
\sup_{0\le f\in L^p(X;l^q)}\frac{\inf_{x\in E}\Hop_\beta f(x)}{\lVert f\rVert_{L^p(X;l^q)}}\\
&= \sup_{0\le f\in L^p(X;l^q)}\Biggl(\left\lVert \frac{f}{\inf_{x\in E}\Hop_\beta f(x)}\right\rVert_{L^p(X;l^q)}^p\Biggr)^{-1/p}\\
&= \Biggl(\inf_{0\le f\in L^p(X;l^q)}\left\lVert \frac{f}{\inf_{x\in E}\Hop_\beta f(x)}\right\rVert_{L^p(X;l^q)}^p\Biggr)^{-1/p}\le \Cp^\beta_{p,q}(E)^{-1/p}\,.
\end{split}
\end{equation}
By \eqref{e.connect}, the Minimax theorem in the following Lemma~\ref{l.m_assumptions}, and~\eqref{e.connect2} we obtain
\begin{align*}
\inf_{\nu\in Z} \frac{\big\lVert \widecheck{\mathcal{H}}_\beta \nu\big\rVert_{L^{p'}(X;l^{q'})}}{\nu(E)}
& = \inf_{\nu\in Z}\sup_{f\in W}\mathcal{E}_\beta(\nu,f)
  = \sup_{f\in W}\inf_{\nu\in Z}\mathcal{E}_\beta(\nu,f)\\
& \le \sup_{f\in W}\inf_{\nu\in Z}\mathcal{E}_\beta(\nu,f)\le \Cp^\beta_{p,q}(E)^{-1/p}\,.
\end{align*}
Hence 
\[
\Cp^\beta_{p,q}(E)^{1/p}\le \sup_{\nu\in Z} \frac{\nu(E)}{\big\lVert \widecheck{\mathcal{H}}_\beta \nu\big\rVert_{L^{p'}(X;l^{q'})}}\le \sup
\Biggl\{\frac{\nu(E)}{\big\lVert \widecheck{\mathcal{H}}_\beta \nu\big\rVert_{L^{p'}(X;l^{q'})}}\,:\, \nu\in\mathcal{M}^+(E) \Biggr\}\,,
\]
as claimed.
\end{proof}

To complete the proof of Theorem~\ref{t.n_cap_dual}, 
we still need to verify that the Minimax theorem is valid in our setting.
The underlying ideas are well known, but we present the details for the sake of completeness.

\begin{lemma}\label{l.m_assumptions}
Under the notation of the proof of Theorem~\textup{\ref{t.n_cap_dual}}, we have
\[
\inf_{\nu\in Z}\sup_{f\in W} \mathcal{E}_\beta(\nu,f)
=
\sup_{f\in W}\inf_{\nu\in Z}\mathcal{E}_\beta(\nu,f)\,.
\]	
\end{lemma}
\begin{proof}
Since $E$ is compact, the Riesz representation
theorem \cite[Theorem 6.19]{MR924157} implies that every $T\in (C(E),\lVert \cdot\rVert_\infty)^*$ is represented by a unique regular Borel measure
$\nu$, in the sense that 
$T(g)=\int_E g\,d\nu$ for every $f\in C(E)$.
The operator norm $\lVert \nu \rVert$ of $\nu$ is the total variation of $\nu$.
Hence, we can regard $Z$ as a subset  of the dual $(C(E), \|\cdot\|_\infty)^*$, and endow $Z$ with the relative weak* topology.

We will apply the  Minimax theorem \cite[Theorem 2.4.1]{MR1411441} to the
functional $\mathcal{E}_\beta$ defined in $Z\times W$.
The sets $Z$ and $W$ are convex and $\mathcal{E}_\beta$
is convex in $Z$ and concave in $W$, see \cite[p.~30]{MR1411441}.
Hence, we need to show that 
\begin{itemize}
\item[(a)] $Z$ is a Hausdorff space,
\item[(b)] $Z$ is compact, and
\item[(c)] for any $f\in W$, the functional $\mathcal{E}_\beta(\cdot,f)$ is lower semicontinuous in $Z$.
\end{itemize}
	
We first prove (a). Let $\nu_1,\nu_2\in Z=\{\nu\in\mathcal{M}^+(E): \nu(E)=1\}$ be two distinct measures. 
By the uniqueness part of the Riesz representation theorem, there exist a real-valued $g\in C(E)$ 
and $\alpha\in\R$ such that
$\int_{E}g\,d\nu_1 < \alpha < \int_{E}g\,d\nu_2$. Then
\[U=\left\{\nu\in Z :  \int_{E}g\,d\nu<\alpha\right\}\quad\text{ and }\quad V=\left\{\nu\in Z : \int_{E}g\,d\nu>\alpha\right\}\] 
are open and disjoint subsets of $Z$, and
$\nu_1\in U$ and $\nu_2\in V$.
Thus $Z$ is a Hausdorff space.

Next we prove (b). Given $\nu\in Z$, the operator norm of $\nu$ on $(C(E), \|\cdot\|_\infty)$ is
$\lVert\nu\rVert=\nu(E)=1$, and therefore
\[
Z\subset \iol{B}(0,1)=\{\nu\in (C(E),\lVert \cdot\rVert_\infty)^*\,:\, \lVert \nu\rVert\le 1\}\,.
\]
Recall that $Z$ is equipped with the relative
weak* topology.
By the Banach--Alaoglu theorem, the closed unit ball $\iol{B}(0,1)$ is
compact in the weak* topology. Hence, in order to show that
$Z$ is compact, it suffices to verify that $Z$ is closed in the weak* topology of $(C(E),\lVert\cdot\rVert_\infty)^*$.
Write 
\[
G=\{g\in C(E)\,:\,g(E)\subset [0,\infty)\}\,.
\]
For every $g\in G$ we define $\Lambda_g\colon (C(E),\lVert\cdot \rVert_\infty)^*\to \mathbb{R}$, $\Lambda_g(\nu)=\int_E g\,d\nu$. 
Then the functionals
$\Lambda_g$ are continuous for every $g\in G$ and 
\[
Z=\Lambda_1^{-1}(\{1\})\cap \bigcap_{g\in G} \Lambda_g^{-1}([0,\infty))\,.
\]
Indeed, above the inclusion
from left to right is immediate. Conversely, if $\nu\in (C(E),\lVert\cdot \rVert_\infty)^*$
belongs to the set on the right-hand side, then
$\nu(E)=1$ and $0\le \int_E g\,d\nu<\infty$ for all $g\in C(E)$ 
satisfying $0\le g< \infty$ in $E$. Hence $\nu$ is a nonnegative linear functional on $C(E)$, and 
\cite[Theorem~2.14]{MR924157} implies that $\nu\in\mathcal{M}^+(E)$.
Thus $\nu\in Z$, and it follows that
$Z$ is closed in $(C(E),\lVert\cdot \rVert_\infty)^*$ since it
is an intersection of closed sets.

It remains to prove condition (c).
Since $E\subset X$ is compact, by  Stone--Weierstrass theorem and Urysohn's lemma, the
Banach space $(C(E),\lVert \cdot\rVert_\infty)$ is separable.
Since $Z\subset (C(E),\lVert \cdot\rVert_\infty)^*$ is weak* compact,
the weak* topology relative to $Z$ is metrizable.
Hence, it suffices to show sequential lower semicontinuity.
To this end, we fix $f\in W$. Recall that
\[
\Hop_\beta f(x)=\sum_{n=n_0}^\infty 2^{-\beta n}\int_{X} \frac{\mathbf{1}_{B(y,2^{-n})}(x)}{\mu(B(y,2^{-n}))}f_n(y)\,d\mu(y)\,,\qquad x\in X.
\]
For each fixed $n\ge n_0$, we define a function $g_n\colon E\to [0,\infty]$, 
\[
g_n(x)=2^{-\beta n}\int_{X} \frac{\mathbf{1}_{B(y,2^{-n})}(x)}{\mu(B(y,2^{-n}))}f_n(y)\,d\mu(y)\,,\qquad x\in E\,.
\]
Let $x\in E$ and let $(x_j)_{j\in\N}$ be a sequence
in $E$ such that $x_j\to x$ in $E$. Then
\[
\mathbf{1}_{B(y,2^{-n})}(x)\le\liminf_{j\to\infty} \mathbf{1}_{B(y,2^{-n})}(x_j)
\]
for all $y\in X$.
Thus, by Fatou's lemma,
\begin{align*}
g_n(x)&\le 2^{-\beta n}\int_{X} \liminf_{j\to\infty}\left(\frac{\mathbf{1}_{B(y,2^{-n})}(x_j)}{\mu(B(y,2^{-n}))}f_n(y)\right)\,d\mu(y)\\
&\le \liminf_{j\to\infty} 2^{-\beta n}\int_{X} \frac{\mathbf{1}_{B(y,2^{-n})}(x_j)}{\mu(B(y,2^{-n}))}f_n(y)\,d\mu(y)= \liminf_{j\to\infty} g_n(x_j)\,.
\end{align*}
Hence, the function $g_n$ is lower semicontinuous in $E$,
and therefore also $\mathcal{H}_\beta f=\sum_{n=n_0}^\infty g_n$ is lower semicontinuous in $E$. 
By~\cite[Proposition~4.2.2]{MR3363168}, there exists a sequence $(h_k)_{k\in\N}$ of continuous functions in $E$ such that $0\le h_1(x)\le h_2(x)\le \dotsb\le \mathcal{H}_{\beta} f(x)$
and $\lim_{k\to\infty} h_k(x)=\mathcal{H}_{\beta} f(x)$ for every $x\in E$.
Now, let $\nu\in Z$ and let $(\nu_i)_{i\in \N}$ be a sequence in $Z$ such that $\nu_i\to \nu$ in $Z$. Then, for all $k\in\N$, 
\begin{align*}
 \int_E h_k(x)\,d\nu(x)&=\lim_{i\to\infty} \int_E h_k(x)\,d\nu_i(x)
  \le \liminf_{i\to\infty} \int_E \mathcal{H}_\beta f(x)\,d\nu_i(x)\\
& =\liminf_{i\to\infty} \int_X \mathcal{H}_\beta f(x)\,d\nu_i(x)\,.
\end{align*}
By using also the monotone convergence theorem, we obtain
\begin{align*}
\mathcal{E}_{\beta}(\nu,f)&=\int_X \mathcal{H}_\beta f(x)\,d\nu(x)=\int_E \mathcal{H}_\beta f(x)\,d\nu(x)
=\lim_{k\to\infty} \int_E h_k(x)\,d\nu(x)
\\&\le \liminf_{i\to\infty} \int_X \mathcal{H}_\beta f(x)\,d\nu_i(x)
=\liminf_{i\to\infty} \mathcal{E}_\beta(\nu_i,f)\,.
\end{align*}
This shows that $\mathcal{E}_{\beta}(\,\cdot\,,f)$ is
sequentially lower semicontinuous, and the proof is complete.
\end{proof}

\section{Muckenhoupt--Wheeden theorem}\label{s.MW}

The original formulation of the celebrated Muckenhoupt--Wheeden Theorem in \cite{MR340523}
bounds the $L^p$-norm of the Riesz potential by the $L^p$-norm of a fractional maximal function
in the (suitably weighted) Euclidean space.
In Theorem~\ref{t.M-W} we present a variant of this result for the 
mixed $L^p(X,l^1)$-norm of the operator $\widecheck{\mathcal{H}}_\beta$.
Theorem~\ref{t.M-W} could actually be obtained as a consequence of a very general 
metric space formulation of the Muckenhoupt--Wheeden Theorem in~\cite{MR1962949}, 
but for the sake of completeness and transparency we give below a more 
organic and self-contained proof. In fact, the application of the general result from \cite{MR1962949} 
would still require rather tedious verification of the various assumptions;
compare for instance to~\cite[Theorem~4.2]{MR3900847}. 
Our proof of Theorem~\ref{t.M-W} follows the general approach in~\cite[Theorem~3.6.1]{MR1411441},
based on a good $\lambda$-inequality, but the present setting requires several modifications. 

For the formulation of Theorem~\ref{t.M-W}, we need fractional maximal functions.

\begin{definition}\label{d.frac_max}
Let $\beta\ge 0$ and let $\nu$ be a nonnegative Borel
measure in $X$.
The fractional maximal function $M_\beta\nu$ is defined by
\begin{equation}\label{e.max_funct_def}
M_\beta\nu(x)=\sup_{0<r\le 2^{-n_0}} r^\beta \frac{\nu(B(x,r))}{\mu(B(x,r))}\,,\qquad x\in X\,.
\end{equation}
When $X$ is unbounded, the supremum is taken over all radii $0<r<\infty$.
When $\beta=0$ we get the centered maximal function
which is denoted by $M\nu=M_0\nu$.
\end{definition}

We need the following technical lemma concerning lower semicontinuity.
Recall that here $\iol{B}(x,2^{-n})$ is the closed
ball centered at $x$.

\begin{lemma}\label{l.lowersc}
Let $\beta>0$ and let $\nu$ be a nonnegative Borel
measure in $X$. Then the function 
\[
\sum_{n=n_0}^{\infty} 2^{-\beta n}\frac{\nu(B(x,2^{-n}))}{\mu(\iol{B}(x,2^{-n}))}\,,\qquad x\in X\,,
\]
is lower semicontinuous in $X$.
\end{lemma}

\begin{proof}
Fix $n\in\Z$ such that $n\ge n_0$.
By the properties of semicontinuous functions, see \cite[p.~362]{MR1726779}, it suffices to show that the nonnegative functions
$f(x)=\nu(B(x,2^{-n}))$ and $g(x)=\mu(\iol{B}(x,2^{-n}))^{-1}$ are lower semicontinuous in $X$.
Let $x\in X$ and let $(x_j)_{n\in\mathbb{N}}$ be a sequence in $X$ such that $x_j\to x$ in $X$. 
By Fatou's lemma,
\begin{align*}
f(x)&=\int_{X} \mathbf{1}_{B(x,2^{-n})}(y)\,d\nu(y)
\le \int_{X} \liminf_{j\to\infty} \mathbf{1}_{B(x_j,2^{-n})}(y)\,d\nu(y)\\
&\le \liminf_{j\to\infty} \int_{X}\mathbf{1}_{B(x_j,2^{-n})}(y)\,d\nu(y)=\liminf_{j\to\infty} f(x_j)\,,
\end{align*}
showing that $f$ is lower semicontinuous in $X$.

Define $h(x)=\mu(\iol{B}(x,2^{-n}))$, $x\in X$. Then,
if $(x_j)_{j\in\N}$ is a sequence converging to $x\in X$ as above, we have
\begin{align*}
h(x)=\int_{X} \mathbf{1}_{\iol{B}(x,2^{-n})}(y)\,d\mu(y)
\ge  \int_{X} \limsup_{j\to\infty} \mathbf{1}_{\iol{B}(x_j,2^{-n})}(y)\,d\mu(y)\,.
\end{align*}
Write $\iol{B}_j =  \iol{B}(x_j,2^{-n})$ and
$B=B(x,M)$, where $M>0$ is chosen so that $\iol{B}_j\subset B$ for all $j\in\N$. Then
$\mathbf{1}_{\iol{B}_j}\le \mathbf{1}_{B}\in L^1(d\mu)$ for all $j\in\N$. 
By applying Fatou's lemma to the nonnegative functions $\mathbf{1}_{B}-\mathbf{1}_{\iol{B}_j}$, 
we obtain
\begin{align*}
\mu(B)-\int_{X} \limsup_{j\to\infty}\mathbf{1}_{\iol{B}_j}(y)\,d\mu(y)
&=\int_{X} \liminf_{j\to\infty} \bigl(\mathbf{1}_{B}-\mathbf{1}_{\iol{B}_j}\bigr)(y)\,d\mu(y)\\
&\le \liminf_{j\to\infty}\int_{X}  \bigl(\mathbf{1}_{B}-\mathbf{1}_{\iol{B}_j}\bigr)(y)\,d\mu(y)\\
&=\mu(B)-\limsup_{j\to\infty}\int_{X}\mathbf{1}_{\iol{B}_j}(y)\,d\mu(y)\,.
\end{align*}
Consequently, we see that
\[
h(x)\ge \int_{X} \limsup_{j\to\infty} \mathbf{1}_{\iol{B}_j}(y)\,d\mu(y)\ge \limsup_{j\to\infty}\int_{X}\mathbf{1}_{\iol{B}_j}(y)\,d\mu(y)=\limsup_{j\to\infty} h(x_j)\,.
\]
Therefore
\[
g(x)=\frac{1}{h(x)}\le \frac{1}{\limsup_{j\to\infty} h(x_j)}=\liminf_{j\to\infty} \frac{1}{h(x_j)}=\liminf_{j\to\infty} g(x_j)\,,
\]
and thus also $g$ is lower semicontinuous in $X$.
\end{proof}

The following proof is an adaptation of the proof of~\cite[Theorem~3.6.1]{MR1411441}.
As usual, it is possible  to generalize Theorem \ref{t.M-W} to Muckenhoupt $A_\infty$ weighted setting.

\begin{theorem}\label{t.M-W}
Let $1<p<\infty$ and $\beta>0$. 
Assume 
that $\mu$ satisfies the quantitative reverse doubling condition~\eqref{e.reverse_doubling} for some exponent $\sigma>\beta$,
and let $\nu$ be a nonnegative Borel
measure in $X$. 
Then there exists
a constant $C=C(c_\sigma,c_\mu,\beta,\sigma,p)$ such that
\begin{equation}\label{e.mconc}
 \big\lVert \widecheck{\mathcal{H}}_\beta \nu\big\rVert_{L^p(X;l^1)} \le C\lVert M_\beta \nu\rVert_{L^p(X)}\,.
\end{equation}
\end{theorem}

\begin{proof}
Observe that we have, for every $x\in X$,
\[
\big\lVert \widecheck{\mathcal{H}}_\beta\nu(x)\big\rVert_{l^1}
=\sum_{n=n_0}^{\infty} 2^{-\beta n}\frac{\nu(B(x,2^{-n}))}{\mu(B(x,2^{-n}))}
\le c_\mu\sum_{n=n_0}^{\infty} 2^{-\beta n}\frac{\nu(B(x,2^{-n}))}{\mu(\iol{B}(x,2^{-n}))}\,.
\]
Throughout the proof, we use the function
\[
h(x)=\sum_{n=n_0}^{\infty} 2^{-\beta n}\frac{\nu(B(x,2^{-n}))}{\mu(\iol{B}(x,2^{-n}))}\,,\qquad x\in X\,.
\]
The advantage of $h$ compared to $\big\lVert \widecheck{\mathcal{H}}_\beta\nu(x)\big\rVert_{l^1}$ is
the lower semicontinuity in $X$, given by Lemma~\ref{l.lowersc}.

If $X$ is bounded, we can make a further reduction concerning the measure $\nu$
by assuming that there
exists $z\in X$ such that 
\begin{equation}\label{e.extra}
\nu\bigl(B(z,\tfrac 1 3\diam(X)\bigr)=0\,.
\end{equation}
Indeed,
if $X$ is bounded, we fix points $x,y\in X$ such that
$d(x,y)\ge (2/3)\diam(X)$. Let $B_x=B(x,(1/3)\diam(X))$
and $B_y=B(y,(1/3)\diam(X))$, whence $B_x\cap B_y=\emptyset$.
Define Borel measures $\nu_x$ and $\nu_y$ in $X$ by setting
$\nu_x(A)=\nu(A\setminus B_x)$ and $\nu_y(A)=\nu(A\setminus B_y)$
for all Borel sets $A\subset X$.
Then
\[
\nu(A)=\nu((A\setminus B_x)\cup (A\setminus B_y))\le \nu_x(A)+\nu_y(A)\,,
\]
for all Borel sets $A\subset X$, and it follows that 
\[
\big\lVert \widecheck{\mathcal{H}}_\beta \nu\big\rVert_{L^p(X;l^1)}\le 
\big\lVert \widecheck{\mathcal{H}}_\beta\nu_x\big\rVert_{L^p(X;l^1)}
+\big\lVert \widecheck{\mathcal{H}}_\beta\nu_y\big\rVert_{L^p(X;l^1)}\,.
\]
Observe also that $M_\beta \nu_x\le M_\beta\nu$ and
$M_\beta \nu_y\le M_\beta\nu$.
Hence, if $X$ is bounded, then
by considering the measures $\nu_x$ and $\nu_y$ instead of $\nu$
we may assume that~\eqref{e.extra} holds for some $z\in X$.

By arguing as in \cite[p.~73]{MR1411441}, 
for the claim~\eqref{e.mconc}
it is sufficient to establish the following good $\lambda$-inequality: 
There exists constants $a=2\cdot c_\mu^5(1+4\cdot 2^{4\beta})>1$, $C_1=C(c_\sigma,c_\mu,\beta,\sigma)>0$ and 
$C_2=C(c_\mu)\ge 1$ such that for
all $\lambda>0$ and all $0<\varepsilon\le (3c_\mu^2)^{-1}$ we have 
\begin{equation}\label{e.goodlambda}
\begin{split}
&\mu(\{x\in X\,:\, h(x)>a\lambda\})\\
&\quad \le C_1\varepsilon\mu(\{x\in X\,:\, h(x)>\lambda\})
+C_2\mu(\{x\in X\,:\, M_\beta \nu(x)>\varepsilon\lambda\})\,.
\end{split}
\end{equation}
To this end, fix $\lambda>0$ and $0<\varepsilon\le (3c_\mu^2)^{-1}$. 
Since the function $h$ is lower semicontinuous in $X$, the level set
$\Omega=\{x\in X \,:\, h(x)> \lambda\}$ is open in $X$.
Since $a>1$, we may assume 
that $\Omega\not=\emptyset$. 

We first consider the case $\Omega=X$.
If $X$ is unbounded, then \eqref{e.goodlambda} holds whenever $C_1>0$
since $\mu(X)=\infty$ by the assumed reverse doubling condition 
and the assumption $\diam(X)=\infty$.
If $X$ is bounded, there exists $z\in X$ such that
\eqref{e.extra} holds. Since $z\in X=\Omega$, we see that
$h(z)>\lambda$. Recall that
by the definition of $n_0$ we have $X=B(x,2^{-n_0})$ for every $x\in X$. 
If $x\in X$, we obtain from \eqref{e.extra} 
that
\begin{align*}
\lambda < h(z)&=\sum_{n=n_0}^{n_0+2} 2^{-\beta n}\frac{\nu(B(z,2^{-n}))}{\mu(\iol{B}(z,2^{-n}))}\le 3c_\mu^2\cdot 2^{-\beta n_0}\frac{\nu(X)}{\mu(X)}\\
&=3c_\mu^2\cdot 2^{-\beta n_0}\frac{\nu(B(x,2^{-n_0}))}{\mu(B(x,2^{-n_0}))}
\le 3c_\mu^2 M_\beta \nu(x)\,.
\end{align*}
It follows that $M_\beta \nu(x)>(3c_\mu^2)^{-1}\lambda$ for every $x\in X$, 
and consequently \eqref{e.goodlambda} holds whenever $C_2\ge 1$ 
since $\varepsilon\le (3c_\mu^2)^{-1}$.
This concludes the proof in the case $\Omega=X$.

In the sequel, we assume that $\Omega\not=\emptyset$ and $X\setminus \Omega\neq\emptyset$. 
By~\cite[Proposition~4.1.15]{MR3363168}, there exists
a countable collection $\mathcal{W}=\{B(x_i,r_i)\}$ of the so-called Whitney balls
in $\Omega$ such that
$\Omega=\bigcup_{i}B(x_i,r_i)$ and
$\sum_{i} \mathbf{1}_{B(x_i,r_i)}\le N=C(c_\mu)$, where
\[
r_i=\tfrac{1}{8}\dist(x_i,X\setminus\Omega)\,.
\]
We write $\mathcal{W}=\mathcal{W}_1\cup \mathcal{W}_2$, where
the collection $\mathcal{W}_1$ consists of all Whitney balls $B_i$ such that
$M_\beta \nu(x)> \varepsilon \lambda$ for all $x\in B_i$,
and $\mathcal{W}_2=\mathcal{W}\setminus \mathcal{W}_1$.
Thus 
\begin{equation}\label{e.z_point}
\{x\in B_i \,:\, M_\beta \nu(x)\le \varepsilon \lambda\}\neq\emptyset
\end{equation}
for all $B_i\in \mathcal{W}_2$.
Since $a>1$ and $\mathbf{1}_\Omega\le \sum_i \mathbf{1}_{B_i}$, we have
\begin{align*}
&\mu(\{x\in X \,:\, h(x)>a\lambda\})
=\int_{\{x\in X\,:\, h(x)>a\lambda\}} \mathbf{1}_\Omega(y)\,d\mu(y)\\
&\qquad \le \sum_{B_i\in\mathcal{W}_1}\int_{\{x\in X\,:\, h(x)>a\lambda\}} \mathbf{1}_{B_i}(y)\,d\mu(y)+\sum_{B_i\in\mathcal{W}_2}\int_{\{x\in X\,:\, h(x)>a\lambda\}} \mathbf{1}_{B_i}(y)\,d\mu(y)\,.
\end{align*}
For the first sum above we obtain
\begin{align*}
&\sum_{B_i\in\mathcal{W}_1}\int_{\{x\in X\,:\, h(x)>a\lambda\}} \mathbf{1}_{B_i}(y)\,d\mu(y)
 \le \sum_{B_i\in\mathcal{W}_1}\mu(B_i)\\
&\qquad \le 
\sum_{B_i\in\mathcal{W}_1}\mu(\{x\in B_i\,:\, M_\beta \nu(x)> \varepsilon \lambda\})
\le \int_{\{x\in X\,:\, M_\beta \nu(x)> \varepsilon \lambda\}}\sum_{B_i\in\mathcal{W}_1}\mathbf{1}_{B_i}(y)\,d\mu(y)\\
&\qquad \le N\mu(\{x\in X\,:\, M_\beta \nu(x)> \varepsilon \lambda\})\,.
\end{align*}
To estimate the second sum, we claim that 
\begin{equation}\label{e.thisc}
\mu(\{x\in B_i\,:\, h(x)>a\lambda\})
\le C(c_\sigma,c_\mu,\beta,\sigma)\varepsilon \mu(B_i)
\end{equation}
holds for all $B_i\in\mathcal{W}_2$. Using~\eqref{e.thisc}, we see that 
\begin{align*}
&\sum_{B_i\in\mathcal{W}_2}\int_{\{x\in X\,:\, h(x)>a\lambda\}} \mathbf{1}_{B_i}(y)\,d\mu(y)\le C(c_\sigma,c_\mu,\beta,\sigma)\varepsilon\sum_{B_i\in\mathcal{W}_2}\mu(B_i)\\
&\qquad=C(c_\sigma,c_\mu,\beta,\sigma)\varepsilon\int_{\{x\in X\,:\, h(x)>  \lambda\}}\sum_{B_i\in\mathcal{W}_2}\mathbf{1}_{B_i}(y)\,d\mu(y)\\
&\qquad \le C(c_\sigma,c_\mu,\beta,\sigma)N \varepsilon\mu(\{x\in X\,:\, h(x)>\lambda\})\,.
\end{align*}
By combining the estimates above, we conclude that the 
good $\lambda$-inequality~\eqref{e.goodlambda} holds with
$C_1=C(c_\sigma,c_\mu,\beta,\sigma)N$
and $C_2=N$, where $N=C(c_\mu)$.

Thus, to complete the proof, we need to show that \eqref{e.thisc}
holds for all $B_i\in \mathcal{W}_2$.
We fix $B_i\in \mathcal{W}_2$ and 
define Borel measures $\nu_1$ and $\nu_2$ by setting $\nu_1(A)=\nu(A\cap 10B_i)$ and 
$\nu_2(A)=\nu(A\setminus 10B_i)$ for all Borel sets $A\subset X$.
We also write, for all $x\in X$, that
\[
h_1(x)=\sum_{n=n_0}^{\infty} 2^{-\beta n}\frac{\nu_1(B(x,2^{-n}))}{\mu(\iol{B}(x,2^{-n}))}\qquad\text{ and }\qquad h_2(x)=\sum_{n=n_0}^{\infty} 2^{-\beta n}\frac{\nu_2(B(x,2^{-n}))}{\mu(\iol{B}(x,2^{-n}))}\,.
\]
Then $\nu=\nu_1+\nu_2$ and $h=h_1+h_2$.
We prove 
inequality \eqref{e.thisc} by showing that
\begin{equation}\label{e.fwanted}
\{x\in B_i\,:\, h(x)>a\lambda\}\subset \left\{x\in B_i\,:\, h_1(x)>a\lambda/2\right\}
\end{equation}
and
\begin{equation}\label{e.twanted}
\mu\left(\left\{x\in B_i\,:\, h_1(x)>a\lambda/2\right\}\right)\le  C(c_\sigma,c_\mu,\beta,\sigma)\varepsilon\mu(B_i)\,. 
\end{equation}

For the inclusion in \eqref{e.fwanted}, we assume that 
$x\in B_i$ is such that $h(x) >a\lambda$, and we first aim to show that $ h_2(x)\le \frac{a}{2}\lambda$. If $10B_i=X$, then 
$h_2(x)=0$. Hence, in the sequel we may assume that $X\setminus 10B_i\not=\emptyset$.
Let $k\ge n_0$ be the largest integer such that 
$B(x,2^{-k})\setminus 10B_i\not=\emptyset$. Then 
$2^{-k}\ge 9r_i> r_i$, and by the definition of $\nu_2$, 
\[
h_2(x)=\sum_{n=n_0}^{k} 2^{-\beta n}\frac{\nu_2(B(x,2^{-n}))}{\mu(\iol{B}(x,2^{-n}))}\,.
\]
Since $9r_i=\frac{9}{8}\dist(x_i,X\setminus\Omega)$,
there exists $w\in 9B_i\setminus \Omega$. Then
$h(w)\le\lambda$ by the definition of the set $\Omega$.

Fix $n\in\Z$ such that $n_0\le n\le k$. Then
\[
\nu_2(B(x,2^{-n}))\le \nu_2(B(w,2^{-n+4}))\le \nu(B(w,2^{-n+4}))
\]
and
\[
\mu(\iol{B}(w,2^{-n+4}))\le \mu(B(x,2^{-n+5}))\le c_\mu^5\mu(\iol{B}(x,2^{-n}))\,.
\]
From these estimates we obtain
\[
2^{-\beta n}\frac{\nu_2(B(x,2^{-n}))}{\mu(\iol{B}(x,2^{-n}))}\le c_\mu^52^{-\beta(n-4)}\frac{\nu(B(w,2^{-(n-4)}))}{\mu(\iol{B}(w,2^{-(n-4)}))}\,,
\]
and so
\[
h_2(x)\le c_\mu^5\sum_{n=n_0}^{k} 2^{-\beta(n-4)}\frac{\nu(B(w,2^{-(n-4)}))}{\mu(\iol{B}(w,2^{-(n-4)}))}\le c_\mu^5(1+4\cdot 2^{4\beta})h(w) 
\le \frac{a}{2}\lambda\,. 
\]
It follows that
\[
a\lambda< h(x) = h_1(x)+h_2(x)\le h_1(x) +\frac{a}{2}\lambda\,,
\]
and thus
$h_1(x)>a\lambda/2$, proving the inclusion in~\eqref{e.fwanted}.

Next we show that \eqref{e.twanted} holds. 
Let $x\in B_i$ be such that $h_1(x)>a\lambda/2$ and choose
$k\in\Z$ satisfying $2^{-k-1}<r_i\le 2^{-k}$. Then
\[
\sum_{n=k+1}^{\infty} 2^{-\beta n}\frac{\nu_1(B(x,2^{-n}))}{\mu(\iol{B}(x,2^{-n}))}
\le M\nu_1(x) \sum_{n=k+1}^{\infty} 2^{-\beta n}\le C(\beta)2^{-\beta k} M\nu_1(x)\le C(\beta)r_i^\beta M\nu_1(x)\,,
\]
where $M$ is the centered maximal operator.
On the other hand, the assumed quantitative reverse doubling condition \eqref{e.reverse_doubling},
with $\sigma>\beta$, gives
\begin{align*}
\sum_{n=n_0}^{k} 2^{-\beta n}\frac{\nu_1(B(x,2^{-n}))}{\mu(\iol{B}(x,2^{-n}))}&\le
\nu_1(X)\sum_{n=n_0}^{k}\frac{2^{-\beta n}}{\mu(B(x,2^{-n}))}\\
&=\nu_1(X)\frac{2^{-\sigma k}}{\mu(B(x,2^{-k}))}\sum_{n=n_0}^{k} 2^{-n(\beta-\sigma)}
\left(\frac{2^{-n}}{2^{-k}}\right)^\sigma \frac{\mu(B(x,2^{-k}))}{\mu(B(x,2^{-n}))}\\
&\le c_\sigma\nu_1(X)\frac{2^{-\sigma k}}{\mu(B(x,2^{-k}))}\sum_{n=-\infty}^{k} 2^{n(\sigma-\beta)}\\
&=C(c_\sigma,\beta,\sigma)\nu_1(X)\frac{2^{-\beta k}}{\mu(B(x,2^{-k}))}\\
&\le C(c_\sigma,\beta,\sigma)\nu_1(X)\frac{r_i^\beta}{\mu(B(x,2^{-k}))}\,.
\end{align*}
Since $\nu_1(X)=\nu(10B_i)$, we obtain
\[
\nu_1(X)\frac{r_i^\beta}{\mu(B(x,2^{-k}))}=r_i^\beta\frac{\nu_1(B(x_i,10r_i))}{\mu(B(x,2^{-k}))}
\le c_\mu^4r_i^\beta\frac{\nu_1(B(x,11r_i))}{\mu(B(x,11 r_i))}\le c_\mu^4 r_i^\beta M\nu_1(x)\,.
\]
By combining the above estimates, we get
\[
a\lambda/2< h_1(x) \le C(c_\sigma,c_\mu,\beta,\sigma)r_i^\beta M\nu_1(x)=\kappa M\nu_1(x)\,, 
\]
where we have written $\kappa=C(c_\sigma,c_\mu,\beta,\sigma)r_i^\beta$.
We conclude that
\begin{equation}\label{e.c1}
\mu\left(\left\{x\in B_i\,:\, h_1(x)>a\lambda/2\right\}\right)
\le \mu(\{x\in X\,:\,M\nu_1(x)>a\lambda/(2\kappa)\})\,.
\end{equation}
The proof of \cite[Lemma~3.12]{MR2867756} shows that 
\begin{equation}\label{e.c2}
\begin{split}
\mu(\{x\in X\,:\,M\nu_1(x)>a\lambda/(2\kappa)\})&\le C(c_\mu) \frac{2\kappa}{a\lambda}\nu_1(X)\\
&\le \frac{C(c_\sigma,c_\mu,\beta,\sigma)}{a \lambda} r_i^\beta\frac{\nu_1(X)}{\mu(B_i)}\mu(B_i)\,.
\end{split}
\end{equation}
Since $B_i\in\mathcal{W}_2$, by \eqref{e.z_point} there exists $z\in B_i$ such that $M_\beta \nu(z)\le \varepsilon \lambda$. 
Taking also into account that $r_i\le (1/8)\diam(X)$ in the bounded case, we obtain
\begin{equation}\label{e.c3}
\begin{split}
r_i^\beta\frac{\nu_1(X)}{\mu(B_i)}&=r_i^\beta\frac{\nu(B(x_i,10r_i))}{\mu(B(x_i,r_i))}
\le c_\mu^4(11r_i)^\beta\frac{\nu(B(z,11r_i))}{\mu(B(z,11 r_i))}\\
&\le \left(\frac{11}{8}\right)^\beta c_\mu^4  M_\beta\nu(z)\le \left(\frac{11}{8}\right)^\beta c_\mu^4 \varepsilon \lambda\,.
\end{split}
\end{equation}
By combining the above estimates \eqref{e.c1}, \eqref{e.c2} and \eqref{e.c3}, we 
conclude that
\[
\mu\left(\left\{x\in B_i\,:\, h_1(x)>a\lambda/2\right\}\right)\le
\frac{C(c_\sigma,c_\mu,\beta,\sigma)}{a} \varepsilon \mu(B_i)\,.
\]
Since $a=2\cdot c_\mu^5(1+4\cdot 2^{4\beta})$, this 
shows that \eqref{e.twanted} holds, and the proof is complete. 
\end{proof}

\section{$q$-independence of capacities}\label{s.indep}

In the previous section, we obtained an upper bound for the 
$L^p(X,l^1)$-norm of $\widecheck{\mathcal{H}}_\beta\nu$
in terms of the $L^p$-norm of $M_\beta\nu$. 
On the other hand, due to the doubling property of $\mu$,
the $L^p(X,l^\infty)$-norm of $\widecheck{\mathcal{H}}_\beta\nu$
gives an easy upper bound for the $L^p$-norm of $M_\beta\nu$;
cf.\ Lemma~\ref{l.m_equiv} and the proof of Lemma~\ref{l.q-independence}. 
Together, these estimates yield the 
main result of this section, Theorem~\ref{t.cnr_ML}, 
which gives the comparability of
the $(\beta,p,q)$-capacities $\Cp^{\beta}_{p,q}(E)$ for all values
of the parameter $1\le q\le \infty$. 
Observe that the Muckenhoupt--Wheeden Theorem~\ref{t.M-W} and the dual
characterization of $\Cp^{\beta}_{p,q}(E)$ in terms of 
$\widecheck{\mathcal{H}}_\beta\nu$ are both crucial 
for this approach, which in Euclidean spaces originally appeared in~\cite{MR1097178}; 
see also~\cite[Sections~3.6 and~4.4]{MR1411441}
and~\cite{MR1066471}.

We begin with a simple two-sided estimate for the fractional maximal function. 

\begin{lemma}\label{l.m_equiv}
Let $\beta>0$ and 
let $\nu$ be a nonnegative Borel measure in $X$. Then
\[
\sup_{n\ge n_0} 2^{-\beta n}\frac{\nu(B(x,2^{-n}))}{\mu(B(x,2^{-n}))} \le M_\beta \nu(x)\le c_\mu\sup_{n\ge n_0} 2^{-\beta n}\frac{\nu(B(x,2^{-n}))}{\mu(B(x,2^{-n}))}
\]
for every $x\in X$, where the supremums are taken over all integers $n$ such that $n\ge n_0$.
\end{lemma} 

\begin{proof}
The first inequality follows from definition \eqref{e.max_funct_def}. For the converse inequality, we let $x\in X$ and $0<r\le 2^{-n_0}$, and let $k\ge n_0$ be an integer such that $2^{-k-1}< r \le 2^{-k}$. Then 
the doubling property \eqref{e.doubling} of $\mu$ implies that
\[
r^\beta \frac{\nu(B(x,r))}{\mu(B(x,r))}\le 2^{-\beta k}\frac{\nu(B(x,2^{-k}))}{\mu(B(x,2^{-k-1}))}\le c_\mu 2^{-\beta k}\frac{\nu(B(x,2^{-k}))}{\mu(B(x,2^{-k}))}\le c_\mu\sup_{n\ge n_0} 2^{-\beta n}\frac{\nu(B(x,2^{-n}))}{\mu(B(x,2^{-n}))}.
\]
The claim follows by taking supremum over all $0<r\le 2^{-n_0}$.
\end{proof}

From Lemma~\ref{l.m_equiv} and the 
Muckenhoupt--Wheeden Theorem~\ref{t.M-W},
we obtain a $q$-independence result for the
operator $\widecheck{\mathcal{H}}_\beta$.

\begin{lemma}\label{l.q-independence}
Let  $1<p<\infty$ and $\beta>0$.
Assume that $\mu$ 
satisfies the quantitative reverse doubling condition~\eqref{e.reverse_doubling} for some exponent $\sigma>\beta$,
and let $\nu$ be a nonnegative Borel measure in $X$.
Then there exists a constant $C=C(c_\sigma,c_\mu,\beta,\sigma,p)>0$ such that
\[
\big\lVert \widecheck{\mathcal{H}}_\beta \nu\big\rVert_{L^p(X;l^1)}
\le C\big\lVert \widecheck{\mathcal{H}}_\beta \nu\big\lVert_{L^p(X;l^{\infty})}\,.
\]
\end{lemma}

\begin{proof}
The Muckenhoupt--Wheeden Theorem \ref{t.M-W} gives
\[
\big\lVert \widecheck{\mathcal{H}}_\beta \nu\big\rVert_{L^p(X;l^1)}\le  C(c_\sigma,c_\mu,\beta,\sigma,p)\lVert M_\beta \nu\rVert_{L^p(X)}\,.
\]
On the other hand, Lemma~\ref{l.m_equiv} implies
\[
\lVert M_\beta \nu\rVert_{L^p(X)}\le c_\mu\big\lVert \widecheck{\mathcal{H}}_\beta \nu\big\rVert_{L^p(X;l^\infty)}\,,
\]
and the claim follows by
combining these two estimates.
\end{proof}

Under the assumptions of the following theorem, the capacities $\Cp^\beta_{p,q}(E)$ and $\Cp^\beta_{p,\tau}(E)$ are
comparable for all $1\le q,\tau\le \infty$;
cf.~\cite[p.~107]{MR1411441}. 
We emphasize that the constant of comparison is independent of $q$ and $\tau$. 

\begin{theorem}\label{t.cnr_ML}
Let  $\beta>0$, $1<p<\infty$, and $1\le q,\tau\le \infty$, and assume that $\mu$ 
satisfies the quantitative reverse doubling condition~\eqref{e.reverse_doubling} for some exponent $\sigma>\beta$.
Then there exists a constant $C=C(c_\sigma,c_\mu,\beta,\sigma,p)>0$ such that
$\Cp^\beta_{p,q}(E)\le C\Cp^\beta_{p,\tau}(E)$
for all compact sets $E\subset X$.
\end{theorem}

\begin{proof}
The dual characterization of the capacity in
Theorem~\ref{t.n_cap_dual} implies that
\[
\Cp^\beta_{p,q}(E)^{1/p}=\sup
\Biggl\{\frac{\nu(E)}{\big\lVert \widecheck{\mathcal{H}}_\beta \nu\big\rVert_{L^{p'}(X;l^{q'})}}\,:\, \nu\in\mathcal{M}^+(E) \Biggr\}
\]
and 
\[
\Cp^\beta_{p,\tau}(E)^{1/p}=\sup
\Biggl\{\frac{\nu(E)}{\big\lVert \widecheck{\mathcal{H}}_\beta \nu\big\rVert_{L^{p'}(X;l^{\tau'})}}\,:\, \nu\in\mathcal{M}^+(E) \Biggr\}\,.
\]
By Lemma~\ref{l.q-independence}, we have
\[
\big\lVert \widecheck{\mathcal{H}}_\beta \nu\big\rVert_{L^{p'}(X;l^{\tau'})}\le \big\lVert \widecheck{\mathcal{H}}_\beta \nu\big\rVert_{L^{p'}(X;l^{1})}\le C\big\lVert \widecheck{\mathcal{H}}_\beta \nu\big\rVert_{L^{p'}(X;l^{\infty})}\le C \big\lVert \widecheck{\mathcal{H}}_\beta \nu\big\rVert_{L^{p'}(X;l^{q'})}
\]
for every $\nu\in\mathcal{M}^+(E)$, 
where $C=C(c_\sigma,c_\mu,\beta,\sigma,p)$. The desired conclusion follows.
\end{proof}

\section{Comparison to Riesz capacity}\label{s.comparison}

In this section we clarify the relations between Riesz capacities and Triebel--Lizorkin capacities.
Riesz capacities are defined in terms of Riesz potentials,
which appear frequently in potential analysis,
see for instance \cite{MR1411441,MR0350027}. 
In metric spaces there are several slightly different
ways to define Riesz potentials. 
The most natural for our purposes is the following variant, 
which appears for instance in \cite{MR1683160,MR1800917,MR3825765}.

\begin{definition}\label{d.RieszPot}
Let $\beta>0$ and let $\nu$ be a nonnegative Borel measure in $X$. The 
Riesz potential $I_\beta \nu$ is defined by
\[
I_\beta \nu(x)=\int_{X\setminus \{x\}}\frac{d(x,y)^\beta}{\mu(B(x,d(x,y)))}\,d\nu(y)\,,\qquad x\in X\,.
\]
If $f\ge 0$ is a Borel function in $X$, we write
$I_\beta f=I_\beta \nu$, where $d\nu=f\,d\mu$.
\end{definition}

Assuming that $f\ge 0$ is a Borel function in $X$ and
$\mu$ satisfies the quantitative reverse doubling condition \eqref{e.reverse_doubling} for some $\sigma>0$,
we have
\begin{equation}\label{e.xaway}
I_\beta f(x)=\int_{X\setminus \{x\}}\frac{d(x,y)^\beta}{\mu(B(x,d(x,y)))}f(y)\,d\mu(y)
=\int_{X}\frac{d(x,y)^\beta}{\mu(B(x,d(x,y)))}f(y)\,d\mu(y)
\end{equation}
for all $x\in X$ since $\mu(\{x\})=0$ for every $x\in X$.

Riesz capacities are well known in Euclidean spaces; 
we refer to \cite{MR1411441,MR946438} and the references therein. In metric spaces, variants of Riesz
capacities have been studied for instance in \cite{CILV,Nuutinen2015TheRC}.
If $\mu$ satisfies the quantitative reverse doubling condition \eqref{e.reverse_doubling} for some $\sigma>0$,
then the following definition coincides with that of \cite[Section 2.3]{CILV}, by \eqref{e.xaway}.

\begin{definition}\label{d.RieszCap}
Let  $\beta>0$ and $1< p<\infty$. The Riesz  
$(\beta,p)$-capacity of a subset $E\subset X$ is 
\[
R_{\beta,p}(E)=\inf \bigl\{ \lVert f\rVert_{L^p(X)}^p\,:\, f\ge 0 \text{ and }I_\beta f\ge 1\text{ on }E \bigr\}\,,
\]
where $f\ge 0$ in the infimum means that $f$ is a nonnegative Borel function in $X$.
\end{definition}

The pointwise inequalities for Riesz potentials in Lemma~\ref{l.riesz_equiv}
will be useful when comparing Riesz capacities and Triebel--Lizorkin capacities.
The proofs of these estimates are rather straightforward computations based on
the doubling and reverse doubling conditions.

\begin{lemma}\label{l.riesz_equiv}
Let $\beta>0$ and
let $\nu$ be a nonnegative Borel measure in $X$. Then
\begin{equation}\label{e.first}
I_\beta \nu(x)\le c_\mu \sum_{n=n_0}^{\infty} 2^{-\beta n}\frac{\nu(B(x,2^{-n})\setminus\{x\})}{\mu(B(x,2^{-n}))}
\end{equation}
for every $x\in X$.

Conversely, if $\mu$ 
satisfies the quantitative reverse doubling condition~\eqref{e.reverse_doubling} 
for some exponent $\sigma>\beta$, then there exists a constant $C=C(c_\sigma,\sigma,\beta)>0$ such that
\begin{equation}\label{e.second}
\sum_{n=n_0}^{\infty}2^{-\beta n}\frac{\nu(B(x,2^{-n})\setminus \{x\})}{\mu(B(x,2^{-n}))} \le CI_\beta \nu(x)
\end{equation}
for every $x\in X$.
\end{lemma} 

\begin{proof}
Fix $x\in X$ and write $A_n=B(x,2^{-n})\setminus B(x,2^{-n-1})$ for every $n\ge n_0$. We have
\begin{equation}\label{e.disc_riesz}
I_\beta \nu(x)=\int_{X\setminus \{x\}}\frac{d(x,y)^\beta}{\mu(B(x,d(x,y)))}\,d\nu(y)=\sum_{n=n_0}^{\infty} \int_{A_n}\frac{d(x,y)^\beta}{\mu(B(x,d(x,y)))}\,d\nu(y)\,.
\end{equation}
The doubling property \eqref{e.doubling} of $\mu$ implies that
\[
\frac{d(x,y)^\beta}{\mu(B(x,d(x,y)))}\le \frac{2^{-\beta n}}{\mu(B(x,2^{-n-1}))}\le c_\mu \frac{2^{-\beta n}}{\mu(B(x,2^{-n}))}
\]
for every $y\in A_n$ and $n\ge n_0$, and thus
\begin{align*}
I_\beta \nu(x)
&\le c_\mu\sum_{n=n_0}^{\infty} \frac{2^{-\beta n}\nu(A_n)}{\mu(B(x,2^{-n}))}\le  c_\mu\sum_{n=n_0}^{\infty} 2^{-\beta n}\frac{\nu(B(x,2^{-n})\setminus \{x\})}{\mu(B(x,2^{-n}))}.
\end{align*}
This proves inequality \eqref{e.first}.

For the proof of \eqref{e.second}, we assume that $\mu$
satisfies the quantitative reverse doubling condition~\eqref{e.reverse_doubling} for some exponent $\sigma>\beta$. 
We estimate
\begin{equation}\label{e.prep}
\begin{split}
&\sum_{n=n_0}^{\infty} 2^{-\beta n}\frac{\nu(B(x,2^{-n})\setminus \{x\})}{\mu(B(x,2^{-n}))}
=\sum_{n=n_0}^{\infty} \int_{B(x,2^{-n})\setminus \{x\}}\frac{2^{-\beta n}}{\mu(B(x,2^{-n}))}d\nu(y)\\
&\qquad = \sum_{n=n_0}^{\infty} \sum_{m=n}^\infty \int_{A_m}\frac{2^{-\beta n}}{\mu(B(x,2^{-n}))}d\nu(y)= \sum_{m=n_0}^{\infty}  \int_{A_m} \sum_{n=n_0}^{m}\frac{2^{-\beta n}}{\mu(B(x,2^{-n}))}d\nu(y).
\end{split}
\end{equation}
Let $m,n\in \Z$ be such that $n_0\le n\le m$ and let $y\in A_m$. Then the quantitative reverse doubling condition \eqref{e.reverse_doubling} implies that
\[
 \frac{\mu(B(x,d(x,y)))}{\mu(B(x,2^{-n}))} \le c_\sigma\biggl(\frac{d(x,y)}{2^{-n}}\biggr)^\sigma.
\]
Hence,
\[
\frac{2^{-\beta n}}{\mu(B(x,2^{-n}))}=2^{n(\sigma-\beta)}\frac{2^{-\sigma n}}{\mu(B(x,2^{-n}))}
\le c_\sigma 2^{n(\sigma-\beta)} \frac{d(x,y)^\sigma}{\mu(B(x,d(x,y)))},
\]
and therefore
\begin{equation}\label{e.prep2}
\begin{split}
\sum_{n=n_0}^{m}\frac{2^{-\beta n}}{\mu(B(x,2^{-n}))}
&\le c_\sigma \frac{d(x,y)^\sigma}{\mu(B(x,d(x,y)))}\sum_{n=-\infty}^m  2^{n(\sigma-\beta)} \\
&\le C(c_\sigma,\sigma,\beta)\frac{d(x,y)^\sigma}{\mu(B(x,d(x,y)))}2^{m(\sigma-\beta)}\\
&\le C(c_\sigma,\sigma,\beta)\frac{d(x,y)^\beta}{\mu(B(x,d(x,y)))}
\end{split}
\end{equation}
for every $m\ge n_0$ and $y\in A_m$.
Combining the estimates~\eqref{e.prep} and~\eqref{e.prep2}
and the identity in~\eqref{e.disc_riesz}, we conclude that
\begin{align*}
\sum_{n=n_0}^{\infty} 2^{-\beta n}\frac{\nu(B(x,2^{-n})\setminus \{x\})}{\mu(B(x,2^{-n}))}
&\le C(c_\sigma,\sigma,\beta)\sum_{m=n_0}^{\infty}  \int_{A_m} \frac{d(x,y)^\beta}{\mu(B(x,d(x,y)))}d\nu(y)\\
&=C(c_\sigma,\sigma,\beta)I_\beta \nu(x)\,,
\end{align*}
as was claimed.
\end{proof}

The following theorem explains the relation between the
Riesz $(\beta,p)$-capacities and the Triebel--Lizorkin $(\beta,p,q)$-capacities. 
Namely, in the case $q=\infty$ these capacities are comparable
for all sets $E\subset X$, under our standard assumptions.
The proof is quite simple, relying on the pointwise
estimates given by Lemma~\ref{l.riesz_equiv}.

\begin{theorem}\label{t.infinity}
Let $\beta>0$ and $1<p<\infty$, and assume that $\mu$ 
satisfies the quantitative reverse doubling condition~\eqref{e.reverse_doubling} 
for some exponent $\sigma>\beta$. Then
\[
c_\mu^{-2p} \Cp^\beta_{p,\infty}(E) \le R_{\beta,p}(E)\le C(c_\mu,c_\sigma,\sigma,\beta,p)\Cp^\beta_{p,\infty}(E)\,,
\]
for all sets $E\subset X$.
\end{theorem}

\begin{proof}
To show the first inequality, we fix a nonnegative Borel function $f$ such that $I_\beta f(x)\ge 1$
for every $x\in E$. We set $h_n=f$ for every $n\ge n_0$ and define $h=(h_n)_{n=n_0}^\infty\ge 0$.
Assuming that $x\in E$, Lemma~\ref{l.riesz_equiv} with the measure $d\nu=f\,d\mu$ gives 
\begin{align*}
1&\le I_\beta f(x) =I_\beta\nu(x)\le c_\mu \sum_{n=n_0}^{\infty} 2^{-\beta n}\frac{\nu(B(x,2^{-n})\setminus\{x\})}{\mu(B(x,2^{-n}))}\\
&\le c_\mu \sum_{n=n_0}^{\infty} 2^{-\beta n}\int_X \frac{\mathbf{1}_{B(x,2^{-n})}(y)}{\mu(B(x,2^{-n}))} f(y)\,d\mu(y)\\
&\le c_\mu^2 \sum_{n=n_0}^{\infty} 2^{-\beta n}\int_X \frac{\mathbf{1}_{B(y,2^{-n})}(x)}{\mu(B(y,2^{-n}))} h_n(y)\,d\mu(y)=\Hop_\beta (c_\mu^2h)(x)\,.
\end{align*}
Thus
\[
\Cp^\beta_{p,\infty}(E)\le \lVert c_\mu^2h\rVert_{L^p(X;l^\infty)}^p=c_\mu^{2p}\lVert h\rVert_{L^p(X;l^\infty)}^p=c_\mu^{2p}\lVert f\rVert_{L^p(X)}^p\,,
\]
and the first inequality of the claim follows by taking infimum over all $f$ as above.

Let then $f=(f_n)_{n=n_0}^\infty$ be a sequence of nonnegative Borel functions in $X$
such that $\Hop_\beta f(x)\ge 1$ for all $x\in E$. We define
$h=\sup_{n\ge n_0} f_n\ge 0$ and fix $x\in E$. 
Using Lemma~\ref{l.riesz_equiv}, with the measure $d\nu=h\,d\mu$, 
and the fact that $\mu(\{x\})=0$ for every $x\in X$,
we obtain
\begin{align*}
1&\le \Hop_\beta f(x)=\sum_{n=n_0}^{\infty} 2^{-\beta n}\int_X \frac{\mathbf{1}_{B(y,2^{-n})}(x)}{\mu(B(y,2^{-n}))} f_n(y)\,d\mu(y)\\
&\le c_\mu \sum_{n=n_0}^{\infty} 2^{-\beta n}\int_{X\setminus \{x\}} \frac{\mathbf{1}_{B(x,2^{-n})}(y)}{\mu(B(x,2^{-n}))} h(y)\,d\mu(y)\\
&=c_\mu \sum_{n=n_0}^{\infty}2^{-\beta n}\frac{\nu(B(x,2^{-n})\setminus \{x\})}{\mu(B(x,2^{-n}))} \le CI_\beta \nu(x)=I_\beta (Ch)(x)\,,
\end{align*}
where $C=C(c_\mu,c_\sigma,\sigma,\beta)$. This implies that
\[
R_{\beta,p}(E)\le \lVert Ch\rVert_{L^p(X)}^p=C^p\lVert h\rVert_{L^p(X)}=C^p\lVert f\rVert_{L^p(X;l^\infty)}^p\,,
\]
and the second inequality of the claim follows by taking infimum over all $f$ as above.
\end{proof}

By the $q$-independence of the $(\beta,p,q)$-capacities for compact sets, given by
Theorem~\ref{t.cnr_ML}, we obtain from Theorem~\ref{t.infinity} 
the comparability of $\Cp^\beta_{p,q}$ and $R_{\beta,p}$ for all $1\le q\le \infty$.
Thus the following result is an immediate consequence of those two theorems.
Nevertheless, we give below also an alternative direct proof,
which is based on the dual characterization of $\Cp^\beta_{p,q}$ and
uses some arguments from \cite{Nuutinen2015TheRC}.

\begin{theorem}\label{t.cnr}
Let  $\beta>0$, $1<p<\infty$, and $1\le q\le \infty$, and assume that $\mu$ 
satisfies the quantitative reverse doubling condition~\eqref{e.reverse_doubling} for some exponent $\sigma>\beta$.
Then there exists a constant $C=C(c_\sigma,c_\mu,\beta,\sigma,p)>0$ such that
$\Cp^\beta_{p,q}(E)\le CR_{\beta,p}(E)$
for all compact sets $E\subset X$.
\end{theorem}

\begin{proof}
Theorem~\ref{t.n_cap_dual} implies that
\begin{equation}\label{e.equiv}
\Cp^\beta_{p,q}(E)^{1/p}=\sup
\Biggl\{\frac{\nu(E)}{\big\lVert \widecheck{\mathcal{H}}_\beta \nu\big\rVert_{L^{p'}(X;l^{q'})}}\,:\, \nu\in\mathcal{M}^+(E) \Biggr\}\,.
\end{equation}
Let $\nu\in\mathcal{M}^+(E)$ and let $f\ge 0$ be such that $I_\beta f\ge 1$ on $E$. 
Then
\begin{align*}
\nu(E)&\le \int_X I_\beta f(x)\,d\nu(x)
=\int_X \int_{X} \mathbf{1}_{X\setminus\{x\}}(y) \frac{d(x,y)^\beta}{\mu(B(x,d(x,y)))}f(y)\,d\mu(y)\,d\nu(x)\\
&\le c_\mu\int_X \int_{X}\mathbf{1}_{X\setminus\{y\}}(x)\frac{d(y,x)^\beta}{\mu(B(y,d(y,x)))}\,d\nu(x)f(y)\,d\mu(y) \\
&=c_\mu \int_X  I_\beta \nu(y) f(y)\,  d\mu(y)\le c_\mu \lVert I_\beta \nu\rVert_{L^{p'}(X)}\lVert f \rVert_{L^p(X)}\,.
\end{align*}
Lemma~\ref{l.riesz_equiv} and Lemma~\ref{l.q-independence} imply that
\begin{align*}
\lVert I_\beta\nu\rVert_{L^{p'}(X)}&\le c_\mu\big\lVert \widecheck{\mathcal{H}}_\beta \nu\big\rVert_{L^{p'}(X;l^1)}
\le C(c_\sigma,c_\mu,\beta,\sigma,p)\big\lVert \widecheck{\mathcal{H}}_\beta \nu\big\rVert_{L^{p'}(X;l^{\infty})}
\\&\le C(c_\sigma,c_\mu,\beta,\sigma,p)\big\lVert \widecheck{\mathcal{H}}_\beta \nu\big\rVert_{L^{p'}(X;l^{q'})}\,.
\end{align*}
Hence, 
\[
\frac{\nu(E)}{\big\lVert \widecheck{\mathcal{H}}_\beta \nu\big\rVert_{L^{p'}(X;l^{q'})}}
\le C(c_\sigma,c_\mu,\beta,\sigma,p)\lVert f \rVert_{L^p(X)}\,,
\]
and by taking infimum over all $f$ as above, we get
\[
\frac{\nu(E)}{\big\lVert \widecheck{\mathcal{H}}_\beta \nu\big\rVert_{L^{p'}(X;l^{q'})}}
\le C(c_\sigma,c_\mu,\beta,\sigma,p) R_{\beta,p}(E)^{1/p}\,.
\]
The claim follows by taking supremum over all 
$\nu\in\mathcal{M}^+(E)$ and using \eqref{e.equiv}.
\end{proof} 

We also record explicitly the following corollary concerning the comparability
of the Riesz capacity and the 
relative Haj{\l}asz--Triebel--Lizorkin $(\beta,p,q)$-capacity.
This result is needed below in our applications concerning capacity density conditions.

\begin{corollary}\label{c.equiv_bp_SEC}
Let $1<p<\infty$, $1<q\le \infty$, $0<\beta <1$, and $\Lambda\ge  2$,
and assume that $\mu$ satisfies the quantitative reverse doubling condition~\eqref{e.reverse_doubling} for some exponent  $\sigma>\beta p$.
Let $B\subset X$ be a ball and let $E\subset \iol{B}$ be a compact set.
Then there exists a constant  $C=C(c_\sigma,p,q,\sigma,\beta,c_\mu,\Lambda)>0$ such that
\begin{equation*}%
\capm{\beta}{p}{q}(E, 2B,\Lambda B)\le   C R_{\beta,p}(E)\,.
\end{equation*}
\end{corollary}

\begin{proof}
This follows by combining Theorems~\ref{t.equiv_bp}
and~\ref{t.cnr}.
\end{proof}

\section{Capacity density}\label{s.density}

We conclude the paper by applying the results from the previous sections
to obtain a characterization of the capacity density condition
for the relative Haj{\l}asz--Triebel--Lizorkin capacity.
This extends the results in \cite{MR946438}, \cite{CILV} and \cite{MR4478471}
to our present setting
and in particular implies that also this density condition
is open-ended (self-improving) in complete geodesic metric spaces; 
see Corollary~\ref{c.riesz-improvement}.

In the proof of the characterization in Theorem~\ref{t.Riesz_and_Hajlasz} we will
need most of the theory that has been established in the paper. In particular, we 
use Corollary~\ref{c.equiv_bp_SEC}, which follows from the comparability results
between the relative Haj{\l}asz--Triebel--Lizorkin $(\beta,p,q)$-capacity,
the Triebel--Lizorkin $(\beta,p,q)$-capacity, and the Riesz $(\beta,p)$-capacity.
In the proofs of the comparability results we applied the discrete $\Lop_\beta$-potentials 
and the $q$-independence of the Triebel--Lizorkin $(\beta,p,q)$-capacities. The proof of the
latter was, in turn, based on the dual characterization of the Triebel--Lizorkin capacities
and the Muckenhoupt--Wheeden Theorem. Finally, we utilize also the 
estimates for the relative $(\beta,p,q)$-capacities from Section~\ref{s.tl_space}.

Besides these results from the present work, one essential ingredient in the proof of 
Theorem~\ref{t.Riesz_and_Hajlasz} is~\cite[Theorem~6.4]{CILV}, which
however requires some additional assumptions on the space~$X$.
The proof of~\cite[Theorem~6.4]{CILV} utilizes a non-trivial characterization of
a Haj{\l}asz capacity density condition in terms of Hausdorff content density conditions
from~\cite[Theorem~9.5]{MR4478471}.

Before the statement of Theorem~\ref{t.Riesz_and_Hajlasz},
we introduce the relevant density conditions.

\begin{definition}\label{d.cap_density}
Let $1< p<\infty$, $1\le q\le\infty$, and $0<\beta<1$.
A closed set $E\subset X$ satisfies the Haj{\l}asz--Triebel--Lizorkin $(\beta,p,q)$-capacity density condition,   
if there are constants $c_0,c_1>0$ and $\Lambda>2$ such that
\begin{equation}\label{e.Hajlaszc capacity density condition}
\capm{\beta}{p}{q}(E\cap \overline{B(x,r)},B(x,2r),B(x,\Lambda r))\ge 
 c_0 \capm{\beta}{p}{q}(\overline{B(x,r)},B(x,2r),B(x,\Lambda r)) 
\end{equation}
for all $x\in E$ and all $0<r<c_1\diam(E)$. 
\end{definition}

By inequality \eqref{e.xaway}, the following Riesz capacity density condition
coincides with the one in \cite[Definition~6.3]{CILV} if $\mu$ satisfies the  
quantitative reverse doubling condition \eqref{e.reverse_doubling} for some $\sigma>0$.

\begin{definition}\label{d.riesz_cap_density}
Let $1\le p<\infty$ and $\beta >0$.
A closed set $E\subset X$ satisfies the Riesz $(\beta,p)$-capacity density condition,
if there is a constant $c_0>0$ such that
\begin{equation}\label{e.Riesz capacity density condition}
R_{\beta,p}(E\cap \overline{B(x,r)})\ge c_0 \: R_{\beta,p}(\overline{B(x,r)}) 
\end{equation}
for all $x\in E$ and all $0<r<(1/8)\diam(E)$.
\end{definition}

We also need a suitable version of the Hausdorff content density condition;
recall Definition~\ref{d.hcc}. %
The following definition coincides with \cite[Definition~5.1]{CILV}.

\begin{definition}
A closed set $E\subset X$ satisfies the Hausdorff content density condition of codimension $d\ge 0$, if  
there is a constant $c_0>0$ such that
\begin{equation}\label{e.hausdorff_content_density}
\Ha^{\mu,d}_r(E\cap \overline{B(x,r)})\ge c_0 \: \Ha^{\mu,d}_r(\overline{B(x,r)})
\end{equation}
for all $x\in E$ and all $0<r<\diam(E)$.
\end{definition}

We are now ready to characterize the Haj{\l}asz--Triebel--Lizorkin capacity density condition.
The implication from (ii) to (iii) in the following theorem relies (via \cite[Theorem 6.4]{CILV})
on~\cite[Theorem 9.5]{MR4478471}. 
This is the most difficult of the implications in the proof.
The bounds $\Lambda \ge 41$ and $c_1\le 1/80$ in (i) and (iv) are convenient due to the use of 
Theorem~\ref{t.HC_bp}, but they have not been optimized. 
Nevertheless, part (iv) shows that the actual value of these constants is not essential. 
Moreover, a straightforward modification of Example \ref{e.ball_example} shows that 
all closed sets $E\subset \R^n$ satisfy condition (i) of Theorem~\ref{t.Riesz_and_Hajlasz} with $\Lambda=2$
when $1<p<\infty$ and $0<\beta<1$ are such that $\beta p<1$.
Not all closed sets satisfy, say, condition (iii) and therefore
the requirement $\Lambda\ge 41$ cannot be replaced with $\Lambda\ge 2$ in part (i)  of Theorem \ref{t.Riesz_and_Hajlasz}.

\begin{theorem}\label{t.Riesz_and_Hajlasz}
Let $1<p<\infty$, $1<q\le \infty$, and $0<\beta< 1$.
Assume that $X$ is a complete geodesic metric space and that $\mu$ satisfies the quantitative reverse doubling condition \eqref{e.reverse_doubling} for some exponent $\sigma>\beta p$. Then the following
conditions are equivalent for all closed sets $E\subset X$:
\begin{enumerate}
\item[\textup{(i)}] $E$ satisfies the Haj{\l}asz--Triebel--Lizorkin $(\beta,p,q)$-capacity density condition \eqref{e.Hajlaszc capacity density condition}, 
with some constants $\Lambda \ge 41$ and $c_1\le 1/80$.
\item[\textup{(ii)}] $E$ satisfies the Riesz $(\beta,p)$-capacity density condition \eqref{e.Riesz capacity density condition}.
\item[\textup{(iii)}] $E$ satisfies the Hausdorff content density condition \eqref{e.hausdorff_content_density} for some $0 < d < \beta p$.
\item[\textup{(iv)}] $E$ satisfies the Haj{\l}asz--Triebel--Lizorkin $(\beta,p,q)$-capacity density condition \eqref{e.Hajlaszc capacity density condition}, 
for all constants $\Lambda \ge 41$ and $c_1\le 1/80$. 
\end{enumerate}
\end{theorem}

\begin{proof}
Assume first that condition (i) holds, with constants $c_0>0$, $\Lambda \ge 41$, and $c_1\le 1/80$ in~\eqref{e.Hajlaszc capacity density condition}.
Fix $x\in E$ and $0<R<(1/8)\diam(E)$, and write $r=8c_1R < c_1\diam(E)$.
The geodesic space $X$ is in particular connected, and
therefore, by \cite[Lemma~5.5]{CILV}, there is a constant $C_1=C(c_\mu,c_\sigma,\sigma,\beta,p)>0$ such that
\begin{equation}\label{e.r1}
R_{\beta,p}(\overline{B(x,R)})\le C_1R^{-\beta p}\mu(B(x,R))
\le C(C_1,c_1, c_\mu)r^{-\beta p}\mu(B(x,r))\,.
\end{equation}
We use Theorem~\ref{t.cap_balls}\,(b)
and inequality~\eqref{e.cap_Eq_comp} to get
\begin{equation}\label{e.r2}
r^{-\beta p}\mu(B(x,r))
\le C(\beta,p, c_\mu,\Lambda)\capm{\beta}{p}{q} (\overline{B(x,r)},B(x,2r),B(x,\Lambda r))\,.
\end{equation}
By combining \eqref{e.r1} and \eqref{e.r2}, we obtain
\[
R_{\beta,p}(\overline{B(x,R)})
\le C(\beta,p,\sigma,c_\sigma, c_\mu, c_1,\Lambda)\capm{\beta}{p}{q}(\overline{B(x,r)},B(x,2r),B(x,\Lambda r))\,.
\]
Since $E$ satisfies the Haj{\l}asz--Triebel--Lizorkin $(\beta,p,q)$-capacity density condition \eqref{e.Hajlaszc capacity density condition}, 
with the constant $c_0>0$, we see that
\[
R_{\beta,p}(\overline{B(x,R)})
\le C(\beta,p,\sigma,c_\sigma, c_\mu,c_0,c_1,\Lambda)\capm{\beta}{p}{q}(E\cap \overline{B(x,r)},B(x,2r),B(x,\Lambda r))\,.
\]
The closed
and bounded set $E\cap \overline{B(x,r)}$ is compact by \cite[Proposition~3.1]{MR2867756},
and so Corollary~\ref{c.equiv_bp_SEC} implies that
\[
R_{\beta,p}(\overline{B(x,R)})\le C(\beta,p,q,\sigma,c_\sigma, c_\mu,c_0,c_1,\Lambda)R_{\beta,p}(E\cap \overline{B(x,r)})\,.
\]
Clearly $R_{\beta,p}(E\cap \overline{B(x,r)})\le R_{\beta,p}(E\cap \overline{B(x,R)})$, and therefore  (ii) holds.

The implication from (ii) to (iii) follows from \cite[Theorem~6.4]{CILV}
and the implication from (iv) to (i) is a triviality.
Hence, it suffices to show that (iii) implies (iv).
To this end, assume that (iii) holds, with a constant $c_0>0$ in \eqref{e.hausdorff_content_density}.
Let $\Lambda \ge 41$ and $c_1\le 1/80$, and fix $x\in E$ and $0<r<c_1 \diam(E)$.
Theorem~\ref{t.cap_balls}\,(a) %
and inequality \eqref{e.cap_Eq_comp} imply that
\begin{equation}\label{e.b1}
\capm{\beta}{p}{q}(\overline{B(x,r)}, B(x,2r),B(x,\Lambda r))
\le C(c_\mu,\beta,p,\Lambda) r^{-\beta p}\mu(B(x,r))\,.
\end{equation}
On the other hand, by the simple estimate in~\cite[Remark 5.2]{CILV}
and the assumed condition (iii),
we obtain
\begin{equation}\label{e.r_hausd}
r^{-d}\mu(B(x,r)) \le \Ha^{\mu,d}_{r}(\overline{B(x,r)})\le
C(c_0)\mathcal{H}^{\mu,d}_{r}(E\cap \overline{B(x,r)})\,. 
\end{equation}
Let $\{B(x_k,r_k)\}_k$ be a countable cover of $E\cap \overline{B(x,r)}$, with $0<r_k\le 5\Lambda r$
for all $k$. If $0<r_k\le r$ for all $k$, then by Definition~\ref{d.hcc} of the Hausdorff content we have
\begin{equation*}%
\mathcal{H}^{\mu,d}_{r}(E\cap \overline{B(x,r)})\le 
\sum_{k} \mu(B(x_k,r_k))\,r_k^{-d}\,,
\end{equation*}
and thus, by~\eqref{e.r_hausd}, 
\begin{equation}\label{e.r_hausd_all}
r^{-d}\mu(B(x,r)) \le 
C(c_0) \sum_{k} \mu(B(x_k,r_k))\,r_k^{-d}\,.
\end{equation}
On the other hand, if $r<r_k$ for some $k$, then we may assume that $r<r_1\le 5\Lambda r$
and that $B(x_1,r_1)\cap \overline{B(x,r)}\neq \emptyset$. 
In this case we obtain, using also the doubling condition~\eqref{e.doubling}, that
\begin{equation}\label{e.r_hausd_one}
\begin{split}
r^{-d}\mu(B(x,r)) & \le C(d,\Lambda)r_1^{-d}\mu(B(x_1,3r_1))
\\&\le C(d,c_\mu,\Lambda)\mu(B(x_1,r_1))\,r_1^{-d}
 \le C(d,c_\mu,\Lambda)\sum_{k} \mu(B(x_k,r_k))\,r_k^{-d}\,.
\end{split}
\end{equation}
By taking infimum over all such covers of $E\cap \overline{B(x,r)}$, we conclude from
\eqref{e.r_hausd_all} and \eqref{e.r_hausd_one} that
\begin{equation}\label{e.r_hausd_bound}
r^{-d}\mu(B(x,r)) 
\le C(c_0,d,c_\mu,\Lambda)\mathcal{H}^{\mu,d}_{5\Lambda r}(E\cap \overline{B(x,r)})\,.
\end{equation}

Since $X$ is connected, the measure $\mu$ satisfies the reverse doubling condition \eqref{e.rev_dbl_decay} for
a constant $0<c_R=C(c_\mu)<1$. 
By using \eqref{e.r_hausd_bound}, Theorem~\ref{t.HC_bp} with $\eta=d/\beta$, 
and inequality \eqref{e.cap_Eq_comp}, we get
\begin{equation}\label{e.b2}
r^{-d}\mu(B(x,r)) \le C(\beta,p,d,c_\mu,c_0,\Lambda) r^{\beta p-d}\capm{\beta}{p}{q}(E\cap \overline{B(x,r)},B(x,2r),B(x,\Lambda r))\,.
\end{equation}
Combining \eqref{e.b1} and \eqref{e.b2} gives
\[
\capm{\beta}{p}{q}(\overline{B(x,r)}, B(x,2r),B(x,\Lambda r))\le
C \capm{\beta}{p}{q}(E\cap \overline{B(x,r)},B(x,2r),B(x,\Lambda r))\,,
\]
where $C=C(\beta,p,d,c_\mu,c_0,\Lambda)$. 
This proves condition (iv), and the proof of the theorem is complete.
\end{proof}

As a corollary of Theorem \ref{t.Riesz_and_Hajlasz},
we obtain for the Haj{\l}asz--Triebel--Lizorkin capacity density condition 
the following self-improvement property,
which is new even in the Euclidean spaces. Similar self-improvement results for
other fractional capacities have been obtained in~\cite{MR946438,CILV,MR4478471}.

\begin{corollary}\label{c.riesz-improvement}
Let $1<p<\infty$, $1<q,\tau\le \infty$, and $0<\beta< 1$.
Assume that $X$ is a complete geodesic metric space and 
that $\mu$ satisfies the quantitative reverse doubling condition \eqref{e.reverse_doubling}
for some exponent $\sigma>\beta p$.
Assume that a closed set $E\subset X$ satisfies the Haj{\l}asz--Triebel--Lizorkin $(\beta,p,q)$-capacity density condition \eqref{e.Hajlaszc capacity density condition}, 
with some constants $\Lambda \ge 41$ and $c_1\le 1/80$.
Then there exists $0<\delta<\min\{\beta,p-1\}$ such that 
$E$ satisfies the Haj{\l}asz--Triebel--Lizorkin $(\gamma,s,\tau)$-capacity density condition \eqref{e.Hajlaszc capacity density condition},
for any $\Lambda \ge 41$ and $c_1\le 1/80$, whenever  
\[
\beta-\delta<\gamma<1\,,\quad   p-\delta<s<\infty\quad \text{   and   }\quad  \sigma > \gamma s\,.
\]
\end{corollary}

\begin{proof}
Theorem~\ref{t.Riesz_and_Hajlasz} implies
that $E$ satisfies the Hausdorff content density condition \eqref{e.hausdorff_content_density} 
for some $0 < d < \beta p$. Notice that this condition is open-ended, so it allows  
us to choose %
$0<\delta<\min\{\beta,p-1\}$  such that $d<(\beta-\delta)(p-\delta)$.
Let $\gamma$ and $s$ be such that 
\[
0<\beta-\delta<\gamma<1\,,\quad 1<p-\delta<s<\infty\quad \text{ and }\quad 
\sigma>\gamma s\,.
\]   
Then $0<d<\gamma s$, and since 
$E$ satisfies the Hausdorff content density condition \eqref{e.hausdorff_content_density} 
of codimension $d$, 
Theorem~\ref{t.Riesz_and_Hajlasz} implies
that $E$ satisfies the Haj{\l}asz--Triebel--Lizorkin $(\gamma,s,\tau)$-capacity 
density condition \eqref{e.Hajlaszc capacity density condition}, 
for all constants $\Lambda \ge 41$ and $c_1\le 1/80$.
\end{proof}

\bibliographystyle{abbrv}

\def\cprime{$'$} \def\cprime{$'$} \def\cprime{$'$}

\setlength{\parindent}{0pt}

\end{document}